\documentclass[12pt]{amsart}

\title{Free oriented extensions of subfactor planar algebras}

\author[ 
S\lowercase{hamindra} G\lowercase{hosh},
C\lowercase{orey} J\lowercase{ones},
M\lowercase{adhav} R\lowercase{eddy}
]
{\bf  
	S\lowercase{hamindra} K\lowercase{umar} G\lowercase{hosh},
	C\lowercase{orey} J\lowercase{ones and}
    B M\lowercase{adhav} R\lowercase{eddy}
}
\date{}
\address{Stat-Math Unit, Indian Statistical Institute, Kolkata, INDIA}
\email{shami@isical.ac.in}
\address{Australian National University, Mathematical Sciences Institute, Canberra, AUSTRALIA}
\email{cormjones88@gmail.com}
\address{Stat-Math Unit, Indian Statistical Institute, Kolkata, INDIA}
\email{madhav0903@gmail.com}

\usepackage{bbm}
\usepackage{mathtools}
\usepackage{cancel}
\usepackage{amsmath}
\usepackage{amsfonts}
\usepackage{latexsym}
\usepackage{amssymb}
\usepackage{mathrsfs}
\usepackage{amscd}
\usepackage{hyperref}
\usepackage{soul}
\usepackage{psfrag}
\usepackage{graphicx}
\usepackage{ulem}
\usepackage{fullpage}
\usepackage[all]{xy}
\usepackage{rotating}
\usepackage{url}
\usepackage{color}
\usepackage{enumerate}
\usepackage{cleveref}
\usepackage{dsfont}

\numberwithin{equation}{section}
\numberwithin{figure}{section}
\theoremstyle{plain}
\newtheorem{thm}{Theorem}[section]
\theoremstyle{plain}
\newtheorem{lem}[thm]{Lemma}
\theoremstyle{remark}
\newtheorem{rem}[thm]{Remark}
\theoremstyle{plain}

\theoremstyle{definition}
\newtheorem{defn}[thm]{Definition}
\theoremstyle{definition}

\theoremstyle{definition}

\theoremstyle{plain}
\newtheorem{prop}[thm]{Proposition}
\theoremstyle{plain}
\newtheorem{fact}[thm]{Fact}

 \newtheorem{problem}[thm]{Problem}
\theoremstyle{definition}
\newcommand{\comments}[1]{}
\newcommand{\ra}{\rightarrow}
\newcommand{\rab}{\rangle}
\newcommand{\lra}{\longrightarrow}

\newcommand{\lab}{\langle}

\newcommand{\mcal}{\mathcal}

\newcommand{\N}{\mathbb N}
\newcommand{\Z}{\mathbb Z}
\newcommand{\C}{\mathbb{C}}
\newcommand{\R}{\mathbb{R}}
\newcommand{\F}{\mathbb{F}}
\newcommand{\mscr}{\mathscr}
\newcommand{\vlon}{\varepsilon}

\newcommand{\oset}{\overset}
\newcommand{\oline}{\overline}
\newcommand{\vphi}{\varphi}

\newcommand{\Irr}{\text{Irr}}


\begin{document}
\global\long\def\vlon{\varepsilon}
\global\long\def\bt{\bowtie}
\global\long\def\ul#1{\underline{#1}}
\global\long\def\ol#1{\overline{#1}}
\global\long\def\norm#1{\left\|{#1}\right\|}
\global\long\def\os#1#2{\overset{#1}{#2}}
\global\long\def\us#1#2{\underset{#1}{#2}}
\global\long\def\ous#1#2#3{\overset{#1}{\underset{#3}{#2}}}
\global\long\def\t#1{\text{#1}}
\global\long\def\lrsuf#1#2#3{\vphantom{#2}_{#1}^{\vphantom{#3}}#2^{#3}}
\global\long\def\tr{\triangleright}
\global\long\def\tl{\triangleleft}
\global\long\def\cc90#1{\begin{sideways}#1\end{sideways}}
\global\long\def\turnne#1{\begin{turn}{45}{#1}\end{turn}}
\global\long\def\turnnw#1{\begin{turn}{135}{#1}\end{turn}}
\global\long\def\turnse#1{\begin{turn}{-45}{#1}\end{turn}}
\global\long\def\turnsw#1{\begin{turn}{-135}{#1}\end{turn}}
\global\long\def\fusion#1#2#3{#1 \os{\textstyle{#2}}{\otimes} #3}

\global\long\def\abs#1{\left|{#1}\right	|}
\global\long\def\red#1{\textcolor{red}{#1}}

\maketitle

\begin{abstract} We show that the restriction functor from oriented factor planar algebras to subfactor planar algebras admits a left adjoint, which we call the free oriented extension functor. We show that for any subfactor planar algebra realized as the standard invariant of a hyperfinite $\rm{II}_1$ subfactor, the projection category of the free oriented extension admits a realization as bimodules of the hyperfinite $\rm{II}_{1}$ factor.

\end{abstract}

\section{Introduction}

Associated to a finite index subfactor $N\subseteq M$ is its \textit{standard invariant}. This can be axiomatized by Ocneanu's paragroups in the finite depth case \cite{O}, and in general by Popa's $\lambda$-lattices \cite{P} or Jones' subfactor planar algebras \cite{J}. The planar algebra approach to standard invariants has become important in classification programs for subfactors \cite{JMS,BJ, AMP}, and has provided useful tools for constructions of exotic fusion categories \cite{BMPS}.

If $N$ and $M$ are both hyperfinte, we can choose an isomorphism $\varphi:N\rightarrow M$ and consider $\mcal H_{\varphi}:={}_NL^2(M)_{\varphi(N)}$ as an $N$-$N$ bimodule, with left action given by the inclusion as usual, but the right action uses the isomorphism $\vphi$ and the right action by $M$.
The bi-category constructed from alternating powers of $\mcal H_{\varphi}$ and $\overline{\mcal H}_{\varphi}$ recovers the standard invariant of the subfactor (see \cite{G,P1}). However, the whole tensor category generated by $\mcal H_{\varphi}$ and $\overline{\mcal H}_{\varphi}$ contains more information than just the subfactor standard invariant. This information is captured by an \textit{oriented} planar algebra $\mathcal{P}_{+}$, whose alternating part is isomorphic to the subfactor planar algebra $\mathcal{P}_{N\subseteq M}$ associated to the original inclusion $N\subseteq M$. We call such a planar algebra an \textit{oriented extension} of $\mathcal{P}_{N\subseteq M}$.

In this note, we introduce a universal such extension called the \textit{free oriented extension} of a subfactor planar algebra, and show that it enjoys a universal property (Theorem \ref{free con any}).
Namely, for any oriented extension there exists a canonical planar subalgebra isomorphic to the free oriented extension.
We can view the free oriented extension as a functor $\mcal{F}$ from the category of subfactor planar algebras and planar algebra homomorphisms to the category of oriented C*-planar algebras and planar algebra homomorphisms. We have a forgetful functor $\mcal{S}$ from the category of oriented C*-planar algebras and planar algebra morphisms to the category of subfactor planar algebras by restricting to the shaded part. 
Thus, we have the following picture.
\[
\left\{\begin{array}{c}
\t{Oriented}\\
\t{C*-planar algebras}
\end{array}\right\}
\begin{array}{c}
\oset {\mcal S}\longrightarrow \\
\oset {\mcal F}\longleftarrow
\end{array}
\left\{\begin{array}{c}
\t{Subfactor}\\
\t{planar algebras}
\end{array}\right\}\]

Further, the universal property implies (infact, by \cite[Chap IV, Theorem 2]{M}, is equivalent to saying) that $\mcal{F}$ is a left adjoint of the functor $\mcal{S}$. 
This also agrees with the philosophy that forgetful functors typically admit left adjoints which are `free'.

It was suggested to us by V.F.R. Jones that the free oriented extension should be related to free products of categories. We show that this is true. Namely, if $Q$ is \textit{any} oriented extension, then the \textit{free} oriented extension is realized inside the free product of the projection category of $Q$ with the category of $\mathbb{Z}$-graded finite dimensional Hilbert spaces (Theorem \ref{orientedrealizationivertible}). We use this result, combined with a result of Vaes, to show the free oriented extension of any hyperfinite $\rm{II}_1$ subfactor planar algebra is realized in the category of bimodules of the hyperfinite $\rm{II}_1$ factor.

\noindent\textbf{Acknowledgement.} The authors would like to thank the Hausdorff Research Institute for Mathematics(HIM), Bonn, where a part of this work was completed during the trimester program on von Neumann algebras (during May-Aug, 2016).

\section{Preliminaries}\label{prelim}
The notion of planar algebra and most of the basic results in the general theory are due to Jones \cite{J}. While many of the results focus on subfactor planar algebras, the arguments apply much more generally. See \cite{J1} for a more general perspective. For explanations closely related to the explanations we give below, see \cite{BHP}. We will give an overview in the interest of being self contained.

\subsection{Oriented planar algebras}\label{sec opa}
Let $\Lambda$ be a non-empty set.
Define another disjoint copy of the set $\Lambda$ via $\ol{ \Lambda } := \left\{ \oline{\lambda} : \lambda \in \Lambda \right\}$.
Consider the free semigroup $W_\Lambda$ (or simply $ W $) of words with letters in $\Lambda \sqcup \ol {\Lambda}$.
There is an obvious involution $ \Lambda\sqcup \ol {\Lambda} \ni \iota  \os - \longmapsto \ol \iota  \in \Lambda\sqcup \ol {\Lambda}$ by defining $ \ol {\ol \lambda} := \lambda$.
This gives rise to the involution $ W \ni w = (\iota _1, \ldots ,\iota _n) \os \ast \longmapsto  {w^*}:=(\oline{\iota }_n, \ldots, \oline{\iota }_1) \in W$.

\begin{defn}{}\label{tan diag} A \textit{$\Lambda$-oriented planar tangle diagram} consists of 
\begin{itemize}
\item a subset $D_0$ of $ \R^2 $ homeomorphic\footnotemark  to the closed unit disc (this is referred to as the `external disc'), and finitely many, mutually non-intersecting subsets, $D_1, \ldots , D_n$ , of $ \t{int} (D_0) $, each of which is homeomorphic\footnotemark[\value{footnote}] to the closed unit disc (referred as `internal discs'),
\item each disc has finitely many marked points on its boundary dividing it into finitely many segments,
\item each disc has a distinguished boundary segment denoted by $\bigstar$,
\item finitely many oriented paths in $D_0\setminus \bigcup _{i=1}^n \t{int} (D_i)$ (referred as `strings') each of which is either a closed loop or has end points at two distinct marked points,
\item the strings exhaust all the marked points (as end points) and are labelled by elements of $\Lambda$ (and not $ \Lambda \sqcup \ol \Lambda $).
\end{itemize}\end{defn}
\footnotetext{If we use diffeomorphisms instead of homeomorphisms (which had been usually done in the literature), the techinical complications in defining the tangle, planar isotopies or even composition of tangles outweigh the actual purpose of introducing tangles.
Hence, we contend ourselves with homeomorphisms.}
\begin{figure}[h]
\psfrag{a}{$ \bigstar $}
\psfrag{b}{$ \bigstar $}
\psfrag{c}{$ \bigstar $}
\psfrag{d}{$ \bigstar $}
\psfrag{e}{\hspace{-0.3em}$ D_1 $}
\psfrag{f}{\hspace{-0.4em}$ D_3 $}
\psfrag{g}{\hspace{-0.4em}$ D_2 $}
\psfrag{h}{\hspace{-0.7em}$ D_0 $}
\psfrag{i}{\hspace{-0.6em}$ \iota _1 $}
\psfrag{j}{\raisebox{-0.3em}{$ \iota _2 $}}
\psfrag{k}{\hspace{-0.4em}$ \iota _3 $}
\psfrag{l}{\hspace{-0.3em}$ \iota _4 $}
\psfrag{m}{\hspace{-0.6em}$ \iota _5 $}
\psfrag{n}{\raisebox{-0.25em}{$ \iota _6 $}}
\psfrag{o}{\hspace{0.7em}$ \iota _8 $}
\psfrag{p}{\raisebox{0.2em}{$ \iota _7 $}}
\psfrag{q}{\raisebox{0.3em}{$ \iota _9 $}}
\includegraphics[scale=0.4]{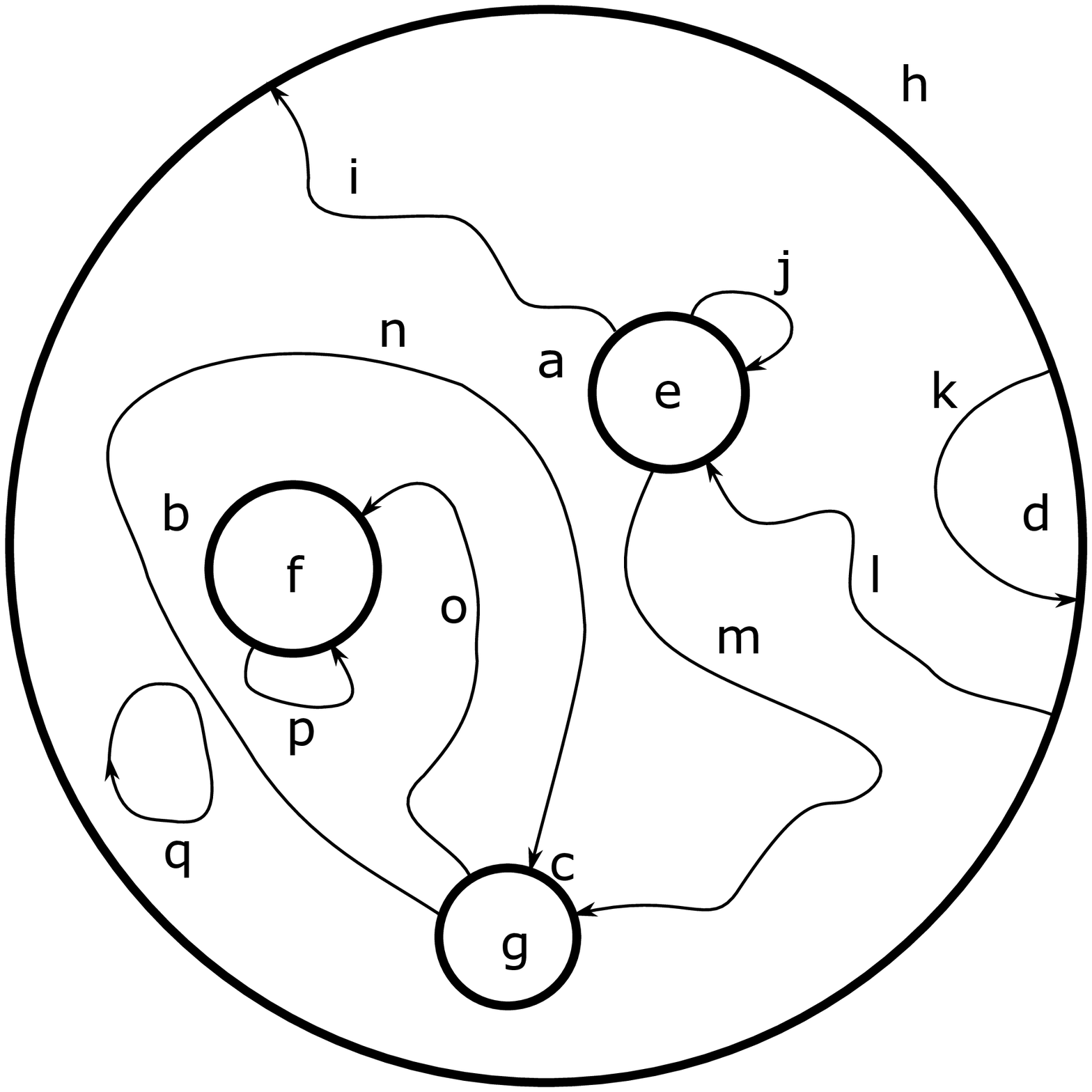}
\end{figure}

Two planar tangle diagrams $ T_1 , T_2 $ are said to be \textit{planar isotopic} if there exists a continuous map $ \vphi : [0,1] \times \R^2  \ra \R^2$ such that $ \vphi_0 = \t{id}_{\R^2} $, $ \vphi_t  $ is a homeomorphism for all $ t \in [0,1] $ and $ \vphi_1 (T_1) = T_2 $ preserving the labeling and the orientation of each string. The planar isotopy class of a $ \Lambda $-oriented planar tangle diagram is called \textit{$ \Lambda $-oriented planar tangle}.

Given such a tangle diagram, to each marked point on the internal disc, there is a string labeled by $\lambda$ which meets this point. To this point we will assign $\lambda$ if the string is oriented away from the internal disc and $ \ol \lambda $ is the string is oriented into the internal disc. For each marked point on the external disc, we have the opposite convention: if the string meeting this point is labeled $\lambda$, then we assign the label $\lambda$ if the string is oriented towards the exterior of the external disc and $\ol \lambda$ is it is oriented towards the interior.

Once this is done, each disc (internal and external) has a unique word in $\Lambda \sqcup\bar{\Lambda}$ attached to it by reading off the letters assigned to marked points starting from $\bigstar$ and moving clockwise along the boundary of the disc.
We call this unique word the \textit{color} of the corresponding disc.
The colors of $D_0$, $D_1$, $D_2$, and $D_3$ in the above tangle diagram are $w_0 = (\iota _3,\bar{\iota }_4,\iota _1,\bar{\iota }_3)$, $w_1 = (\iota _1,\iota _2,\bar{\iota }_2,\bar{\iota }_4,\iota _5)$, $w_2 = (\bar{\iota }_5,\iota _6,\iota _8, \bar{\iota }_6)$ and $w_3 = (\bar{\iota }_8,\bar{\iota }_7,\iota _7)$ respectively.
Note that planar isotopy does not affect the colors of the disc.
If tangle $ T $ has internal discs with colors  $ v_1,v_2,\ldots v_n $ and external disc with color $v_{0}$, then we denote this situation by $ T: (v_1,v_2,\ldots v_n) \rightarrow v_0 $; the set of all such tangles will be denoted by $ \mscr T_{(v_1,\ldots ,v_n);v_0} $. 
For example, in the tangle shown above we have $T:(w_1,w_2,w_3 )\ra w_0$.
If there is no internal disc in $ T $, then it is denoted by $T: \varnothing \ra v_0$; the set of all such tangles will be referred as $ \mscr T_{\varnothing;v_0} $.
Further, $ \mscr T_{v_0} = \mscr T^\Lambda_{v_0} $ will denote the set of all tangles each of whose external disc has color $ v_0 $.
Composition of two $ \Lambda $-oriented tangles is defined in exactly in the same way as in Jones' shaded planar tangles. (See \cite{J}).

In order to draw a picture of an oriented tangle, it will be convenient to represent a collection of parallel strings in any portion of the tangle by a single oriented string labelled by the word in $ W $, constructed from the letters labelling the individual strings along with their orientations.
For example, if $ w = (\iota _1, \ol \iota _2 , \ol \iota _3) $ where $ \iota _1, \iota _2, \iota _3 \in \Lambda $, then \;
\psfrag{a}{\hspace{-0.3em}\raisebox{-0.2em}{$ w $}}
\raisebox{-1.45em}{\includegraphics[scale=0.2]{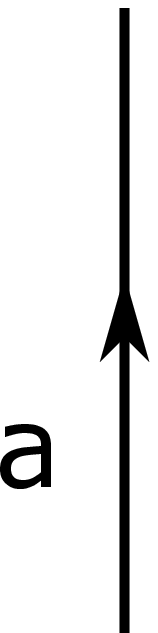}}
\; will represent
\psfrag{a}{\hspace{-0.6em}$ \iota _1 $}
\psfrag{b}{\hspace{-0.7em}$ \iota _2 $}
\psfrag{c}{\hspace{-0.7em}$ \iota _3 $}
\hspace{0.5em}\raisebox{-1.45em}{\includegraphics[scale=0.2]{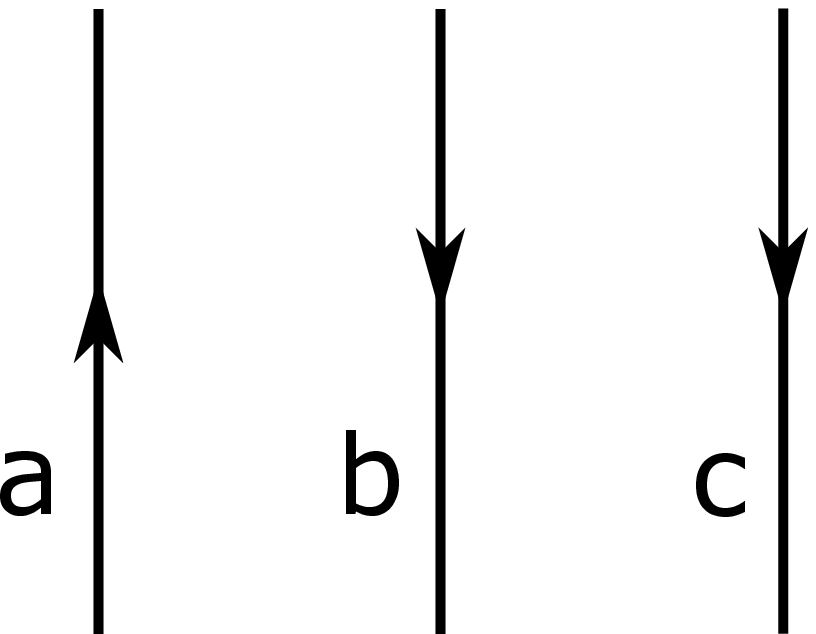}}.
With this convention, note that \;
\psfrag{a}{\hspace{-0.3em}\raisebox{-0.2em}{$ w $}}
\raisebox{-1.45em}{\includegraphics[scale=0.2]{figures/prelim/1.eps}}
\;$ = $\;
\psfrag{a}{\hspace{-0.47em}\raisebox{-0.22em}{$w^* $}}
\raisebox{-1.45em}{\includegraphics[scale=0.2]{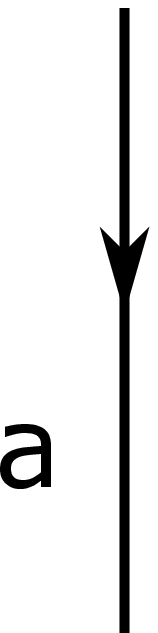}}.
\comments{If $ w $ is a word with letters $ \iota _i \in \Lambda\sqcup \bar\Lambda $, by a string labelled by $ w $ we mean $ \abs w $- many parallel strings each of which is labelled by $ \iota _i $, $ i=1,2,\ldots n $, that is, 
$\; \psfrag{a}{\hspace{-0.3em}\raisebox{-0.2em}{$ w $}}
\raisebox{-1.45em}{\includegraphics[scale=0.2]{figures/prelim/1.eps}}= 
\psfrag{a}{\hspace{-0.6em}$ \iota _1 $}
\psfrag{b}{\hspace{-0.7em}$ \iota _2 $}
\psfrag{c}{\hspace{-0.7em}$ \iota _n $}
\hspace{0.5em}\raisebox{-1.45em}{\includegraphics[scale=0.2]{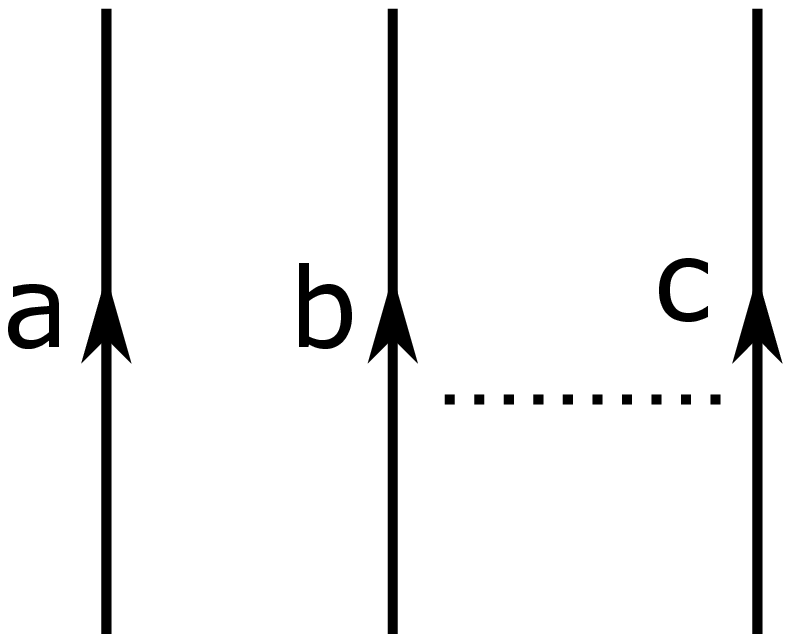}}
$
}

Before we proceed to more definitions, we describe some useful tangles.
Here we will often replace discs by rectangles, so there is a natural ``source" and ``target" associated to these diagrams.

Let $ w,w_1,w_2 \in W$.
\begin{itemize}
\item \textit{Identity tangle:} \[I_w :=\psfrag{a}{$ w $}
	\psfrag{c}{$ \bigstar $}
	\psfrag{d}{$ \bigstar $}
	\psfrag{e}{\hspace{-0.3em}\raisebox{-0.1em}{$ D_1 $}}
	\raisebox{-1.4em}{\includegraphics[scale=0.22]{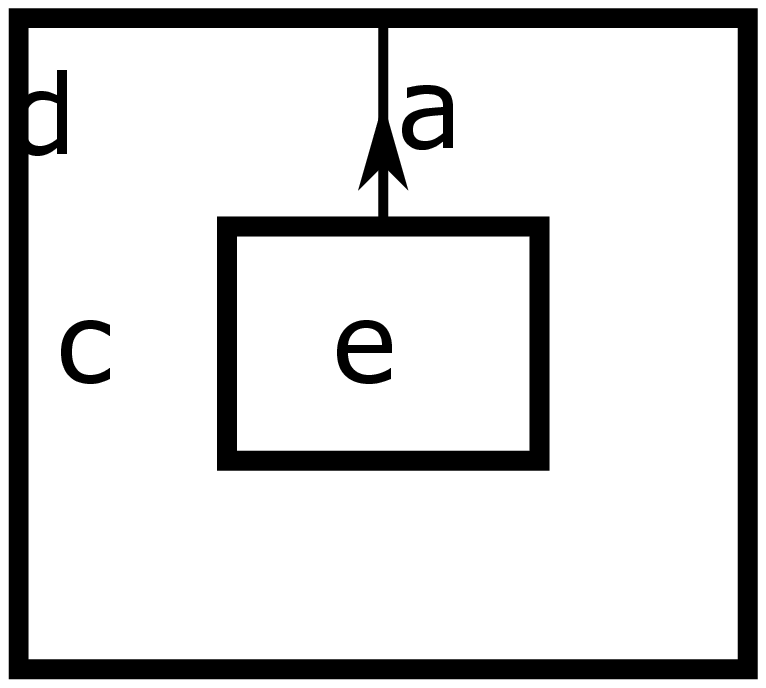}}: w \rightarrow w \] 

\item \textit{Unit tangle:} \[1_w\coloneqq\psfrag{a}{$ w $}
	\psfrag{b}{$ \bigstar $}
	\raisebox{-1.4em}{\includegraphics[scale=0.22]{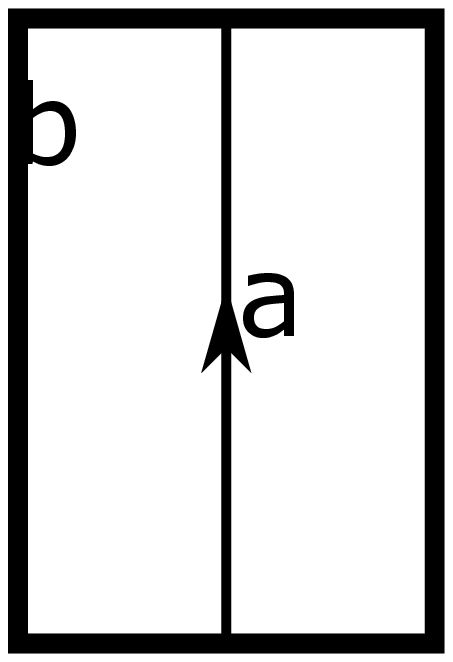}}: \varnothing \rightarrow w{w^*}\]

\item \textit{Multiplication tangle:} \[M^w_{w_1,w_2} :=
	\psfrag{a}{\hspace{-0.13em}\raisebox{-0.1em}{$ \bigstar $}}
	\psfrag{b}{$ w_1 $}
	\psfrag{c}{$ w $}
	\psfrag{d}{\raisebox{-0.1em}{$ w_2 $}}
	\psfrag{e}{\hspace{-0.5em}\raisebox{-0.2em}{$ \bigstar$}}
	\psfrag{f}{\hspace{-0.5em}\raisebox{0em}{$ \bigstar $}}
	\psfrag{g}{\hspace{-0.4em}\raisebox{-0.1em}{$ D_1 $}}
	\psfrag{h}{\hspace{-0.4em}\raisebox{-0.12em}{$ D_2 $}}
	\raisebox{-2.5em}{\includegraphics[scale=0.22]{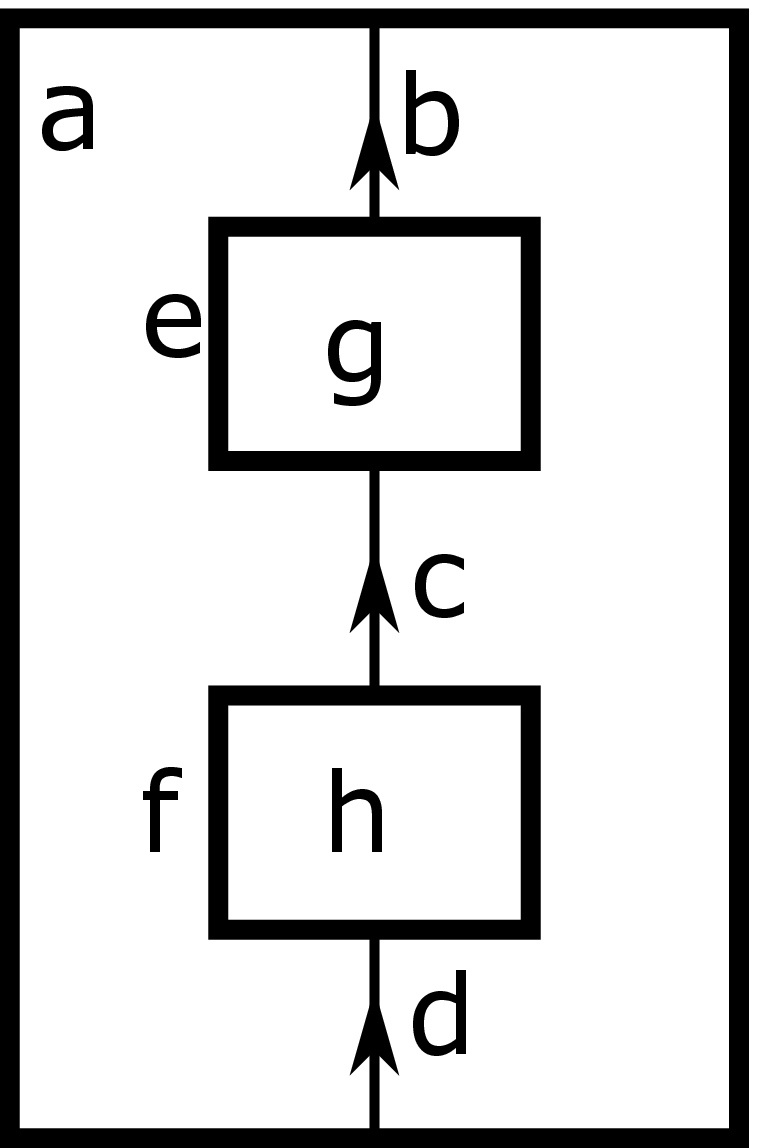}}
	: (w_1 w^*\:,\: w  w^*_2) \rightarrow  w_2w^*_1 \]

\item \textit{Inner product tangle:} \[H_w :=\psfrag{a}{\hspace{-0.2em}\raisebox{-0.1em}{$ \bigstar $}}
	\psfrag{c}{$ w $}
	\psfrag{e}{\hspace{-0.5em}\raisebox{-0.2em}{$ \bigstar $}}
	\psfrag{f}{\hspace{-0.5em}\raisebox{-0.1em}{$ \bigstar $}}
	\psfrag{g}{\hspace{-0.35em}\raisebox{-0.1em}{$ D_2 $}}
	\psfrag{h}{\hspace{-0.35em}\raisebox{-0.1em}{$ D_1 $}}
	\raisebox{-2.5em}{\includegraphics[scale=0.22]{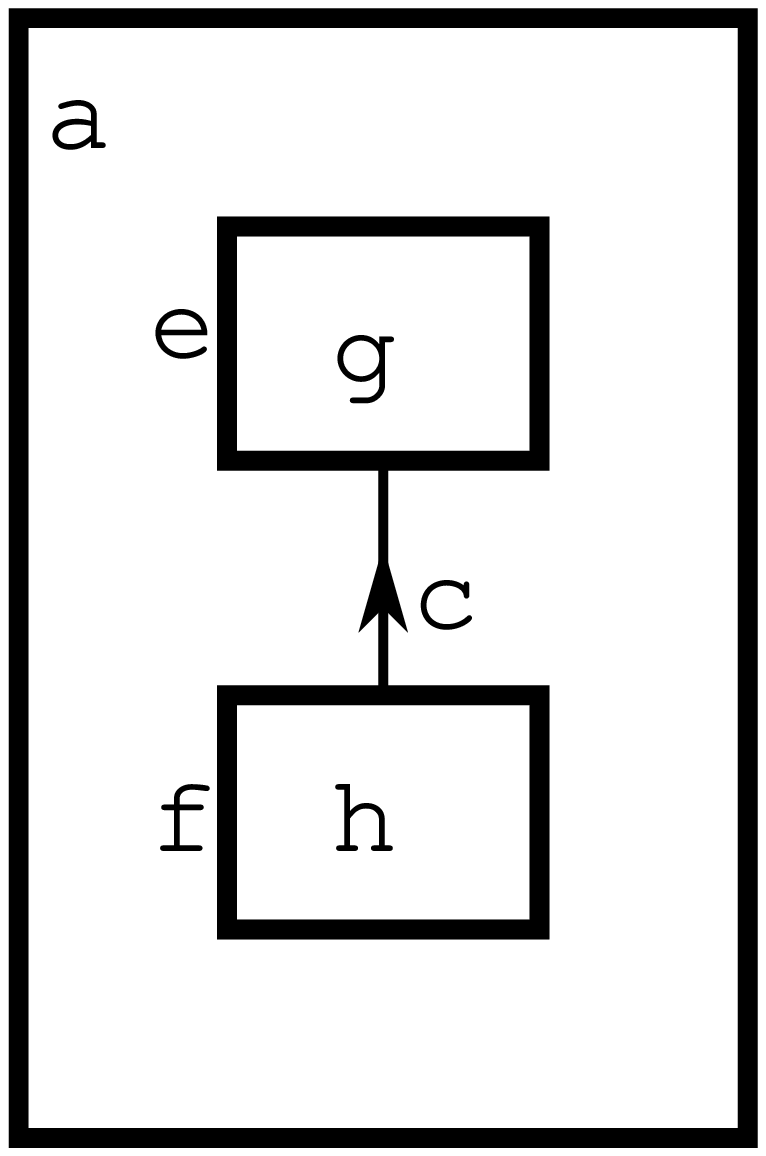}}
: (w,{w^*}) \rightarrow \varnothing\]

\item \textit{Rotation tangle:} \[\rho_{w_1,w_2}:=
\psfrag{a}{\hspace{-0.2em}\raisebox{-0.1em}{$ D_1 $}}
\psfrag{b}{\hspace{-0.5em}\raisebox{0.1em}{$ \bigstar$}}
\psfrag{c}{\raisebox{0em}{$ w_2 $}}
\psfrag{d}{\raisebox{-0.3em}{$ \bigstar $}}
\psfrag{e}{\raisebox{0.1em}{$ w_1 $}}
\raisebox{-1.7em}{\includegraphics[scale=0.25]{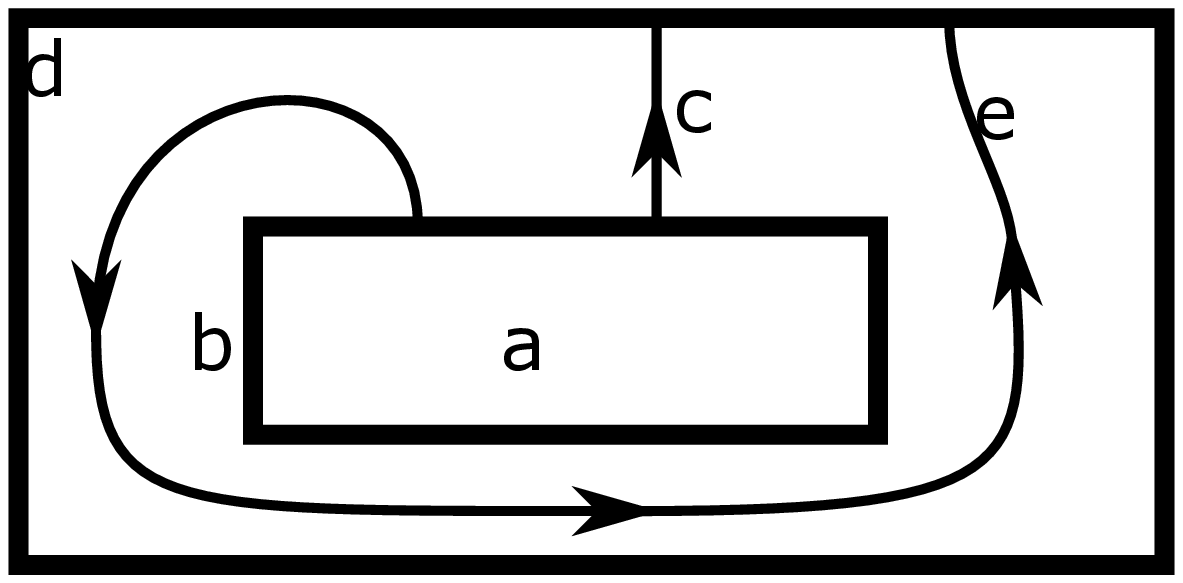}}: (w_1w_2)\rightarrow (w_2w_1)\]
\end{itemize}

When $w=w_1=w_2$, we denote the multiplication tangle simply by $M_w$.

\begin{defn}{} \label{opa}
A \textit{$\Lambda$-oriented planar algebra} $P$ consists of collection of complex vector spaces $\{P_w\}_{w \in W}$ and for every tangle $T : (w_1,\ldots , w_n) \rightarrow w_0$, we have a multi-linear map $P_T$ : $P_{w_1}\times \cdots\times P_{w_n} \rightarrow P_{w_0}$ satisfying the following conditions:

\begin{enumerate}
\item For tangles $ S:(w_1,\ldots,w_n) \ra w_0 $ and $ T:(u_{1},\ldots,u_m)\rightarrow w_{j} $ we have
\[ P_{S \us{D_j} \circ T} = P_S \circ ( \t{id}_{P_{w_1}} \times \cdots \times \t{id}_{P_{w_{j-1}}} \times P_T  \times \t{id}_{P_{w_{j+1}}} \times \cdots \times \t{id}_{P_{w_{n}}} )\]

\medskip

\item $P_{I_w}=\t {id}_{P_w}$ for all words $w$

\medskip

\item For $ T:(w_1,\ldots,w_n) \ra w_0 $, any collection of $ x_j \in P_{w_j}$ for $ 1\leq j \leq n $, and $ \sigma \in S_n $, we have 
$$ P_{\sigma(T)}  (x_{\sigma^{-1} (1)} , \ldots , x_{\sigma^{-1} (n)} ) = P_T (x_1 ,\ldots, x_n),$$ 

\noindent where the tangle $ \sigma(T):(w_{\sigma^{-1}(1)}, \ldots ,w_{\sigma^{-1}(n)}) \ra w_0 $ is obtained by renaming the $ j $-th internal disc $ D_j $ in $ T $ as $ D_{\sigma(j)} $ in $ \sigma (T) $ for $ 1\leq j \leq n $.
\end{enumerate}
\end{defn}

In the above definition, we adopt the convention that $T:\varnothing\rightarrow w$ should give a map $P_{T}: \C\rightarrow P_{w}$. This is consistent with the convention that the tensor power of vectors spaces over the empty set is the scalar field. Note that $\varnothing$ is a valid word, so there is a distinction between a tangle  $T:(\varnothing)\rightarrow w$ and $T:\varnothing\rightarrow w$. The first has an internal disc with no strings attached to it, while the second has no internal disc at all.

As in Jones' planar algebras, the multiplication and unit tangles equip $\{P_{v{v^*}}\}_{v \in W}$ unital associative algebras. We often write $ x\cdot y $ (or simply $ xy $) for $ P_{M_{v_1,v_2}^v}(x,y) $ whenever $ x$ and $y  $ are in appropriate spaces to make sense of the action.

There is also a $ * $-structure on the space of tangles.
For $T : (w_1, \ldots,w_n) \rightarrow w_0$ the tangle $T^*: (w^*_1, \ldots, w^*_n) \rightarrow w^*_0$ is obtained by reflecting $T$ along any straight line where the numbering of the internal discs (if any) and $ \Lambda $-labels of the strings are induced by the reflection from $ T $ whereas the orientation of each string is reversed after reflection.
Clearly $ *  $ is an involution.
\begin{defn}{} \label{pa types}
A $\Lambda$-oriented planar algebra $P$ is said to be 
\begin{enumerate}
\item \textit{connected} if $\t{dim} (P_{\varnothing}) = 1$.
\item \textit{finite dimensional} if $\t{dim} (P_w) < \infty$ for every $w \in W$.
\item a \textit{$*$-planar algebra} if there exists conjugate linear involutions $\left\{*_w : P_w \rightarrow P_{{w^*}}\right\}_{ w \in W }$ such that $\left[P_T(x_1,\ldots,x_n)\right]^* = P_{T^*}(x_1^*,  \ldots, x_n^*)$.
\item \textit{positive}, if $ P $ is a connected $*$-planar algebra and the sesquilinear form $P_{H_w} \circ (\t {id}_{P_w} \times \ast_w)$ is positive definite (and thereby gives an inner product on $ P_w $) for every $w \in W$.
\item \textit{spherical} if $ P $ is connected and actions of any two spherically isotopic tangles are identical.
\end{enumerate}
\end{defn}
In this article, we mainly focus on $ \Lambda $-oriented planar algebra which are connected,  finite-dimensional, positive and spherical; we will refer these as \textit{$\Lambda$- oriented factor planar algebra}. Note that this is more flexible than the definition given in \cite{BHP}, where they assume $ \lambda =\bar{\lambda} $ for every $ \lambda \in\Lambda $.

The reason we call ours factor planar algebras as well is as follows: a $\rm{II}_{1}$ factor and a collection of bifinite bimodules $\Lambda$, one can construct a $\Lambda$-oriented planar algebra as described in the next section. By the results of \cite{BHP}, every $\Lambda$-oriented planar algebra can be realized this way for some $\rm{II}_{1}$ factor. 

We now proceed to discuss morphisms of planar algebras.

\begin{defn}
Let $ P $ (resp., $ P' $) be a $ \Lambda $- (resp., $ \Lambda' $-) oriented planar algebra and $\varphi: \Lambda\rightarrow \Lambda' $ be any map.
If $ W_\Lambda$ and $W_{\Lambda'}  $ denote the free semigroups of words with letters in $ \Lambda\sqcup\ol{\Lambda} $ and $\Lambda'\sqcup\ol{\Lambda'} $ respectively, then $ \vphi $ extends to a homomorphism $ \vphi: W_\Lambda \ra W_{\Lambda'} $ by setting $ \varphi(\ol\lambda) := \ol{\varphi(\lambda)}$ for $ \lambda \in \Lambda $.
If $ T $ is a $ \Lambda $-oriented $ w $-tangle, then replacing labels assigned to strings by its corresponding image under $ \varphi $, we get a unique $ \Lambda' $-oriented tangle $ \varphi(T) $.
Thus $ \varphi $ can also be seen as a map from all $ \Lambda $-oriented tangles to $ \Lambda' $-oriented tangles which preserves composition and identity.
A \textit{homomorphism} $ \varphi:P\rightarrow P' $ consists of a map $ \varphi:\Lambda\rightarrow \Lambda' $ along with a collection of linear maps $ \varphi_w :P_w\rightarrow P'_{\varphi(w)}$ for each $ w \in W_\Lambda $ such that the action of oriented tangles is preserved i.e., for every $ \Lambda $-oriented tangle $ T :(w_1,\ldots,w_n)\rightarrow w$, we have $ \varphi_w(P_T(x_1,\ldots,x_n))=P'_{\varphi(T)}(\varphi_{w_1}(x_1),\ldots, \varphi_{w_n}(x_n)) $ for every $ x_i \in P^{\Lambda}_{w_i} $, $ i=1,2,\ldots,n$.
It is said to be an \textit{isomorphism} if all the maps $ \varphi, \varphi_w, w\in W $ are bijections.

If $ P$ and $ P' $ are $ * $-planar algebras, then $ \vphi $ is called a \textit{$ * $-homomorphism} if each $ \varphi_w $ preserves the $ * $-structure.
\end{defn}
\begin{rem}
Any $ * $-homomorphism between two oriented factor planar algebras will be automatically injective (cf. \cite{J}). 
\end{rem}
\subsection{Some basic facts of C*-categories}\label{basics-cat}
Let $ \mcal C $ be a C*-category.
For $ X,Y, X_i  \in \t{Obj} (\mcal C)$, $ 1 \leq i \leq n $, we say:\\
(i) \textit{$ Y $ is a subobject of $ X $} if the morphism space $\mcal C(Y,X) $ contains an isometry,\\
(ii) a projection \textit{$ p \in \mcal C (X,X) $ factors through $ Y $} if there exists an isometry $ u \in \mcal C (Y,X) $ such that $ p = u u^* $,\\
(iii) \textit{$ X $ is a direct sum of $ \{X_i\}^n_{i=1} $} if for all $ 1 \leq i \leq n $, there exists an isometry $ u_i \in \mcal C(X_i , X) $ such that $ 1_X = \displaystyle \sum^n_{i=1} u_i u^*_i $.

In general, $ \mcal C $ may neither be closed under direct sums nor have every projection factoring through a subobject.
However, we have a C*-category $ \mcal K (\mcal C) $ which we refer as the \textit{unitary Karoubi envelope of $ \mcal C $}, such that the following holds:

(1) there exists a fully faithful $ * $-functor $ \iota : \mcal C \ra \mcal K (\mcal C) $ which is isometry on morphisms,

(2) every projection in $ \mcal K (\mcal C) $ factors through a subobject,

(3) $ \mcal K (\mcal C) $ is closed under direct sums, and

(4) every $ Z \in \t{Obj} (\mcal K (\mcal C)) $ appears as a subobject of direct sum of objects of the form $ \iota (X) $, $ X \in \t {Obj} (\mcal C) $.\\
Moreover, the pair $ \left( \mcal K (\mcal C), \iota \right)$ satisfies the universal property:\\
for every pair $ (\mcal D , \sigma) $ where $ \mcal D $ is a C*-category closed under direct sums and having every projection factoring through a subobject, and $ \sigma :\mcal C \ra \mcal D $ is a C*-functor, there exists a C*-functor $ \widetilde \sigma : \mcal K (\mcal C) \ra \mcal D $ such that $ \sigma $ is equivalent to $ \widetilde \sigma \circ \iota $ via a unique natural unitary.

It is rather easy to achieve conditions (1) and (2) by considering the \textit{unitary idempotent completion} of $ \mcal C $, denoted by $\t {proj} (\mcal C )$, whose objects are pairs $ (X , p) $ fo $ X \in \t {Obj} (\mcal C) $ and project $ p \in \mcal C(X,X) $.
The morphism space from $ (X,p) $ to $ (Y,q) $ consists of those $ f \in \mcal C(X,Y) $ satisfying $q \circ f \circ p = f $.
The *-structure on $\t{proj}(\mcal{C})$ is simply induced from $\mcal{C}$.
However, $ \t{proj} (\mcal C) $ might still be far from being closed under direct sums.
We denote the canonical functor $X \mapsto (X,1_X)$ for $ X \in \t{Obj} (\mcal C) $ by $ \alpha^{\mcal C} : \mcal C \ra \t{proj}(\mcal C)$.
Any C*-functor $ F:\mcal C \ra \mcal D $ between C*-categories, induces a canonical C*-functor between the unitary idempotent completions by simply applying $ F $; we denote this by proj$(F):  $proj$ (\mcal C) \ra $proj$ (\mcal D) $. Further, if $ F $ is monoidal, then $\t{proj}(F) $ is also.

\begin{fact}\label{projcat}
	\begin{enumerate}
		\item If $\mathcal{C}$ is a strict C*-tensor category, then $\t{proj}(\mcal C)$ also inherits this structure by extending the tensor product appropriately.
		Note that $ \alpha^{\mcal C} (X \otimes Y) = \alpha^{\mcal C} (X) \otimes \alpha^{\mcal C} (Y)$, and indeed $ \alpha^{\mcal C} $ is trivially monoidal.
		
		\item If $ \mcal C $ is a semi-simple C*-category, then so is proj$ (\mcal C) $, and $\alpha^{\mcal C}: \mcal C \ra \t {proj} (\mcal C) $ is an equivalence.
		
		\item If $ F : \mcal C \ra \mcal D $ is a C*-functor, then $ \t {proj}(F) \circ \alpha^{\mcal C} = \alpha^{\mcal D} \circ F $.
	\end{enumerate}
\end{fact}

Throughout this article, we use monoidal C*-functors between C*-tensor categories where the tensor preserving properties of the functors are implemented by natural unitaries; we refer such functors as \textit{unitary tensor functors}.

Now, we briefly go over a constructive description of the Karoubi envelope; for details, see \cite{JP, GLR}.
Let $ \mcal C $ be a C*-category which is not necessarily semisimple.
By the \textit{unitary direct sum completion of $ \mcal C $}, we mean a category  $\mathcal{C}^{\oplus}$ with objects as formal unitary (finite) direct sums of objects of $ \mcal C $.
Morphisms between formal direct sums are matrices whose entries are morphisms between the corresponding objects.
The axioms of a C*-category guarantee this category is again C*.
If $\mathcal{C}$ is a C*-tensor category, then $\mathcal{C}^{\oplus}$ also inherits this structure by extending the tensor product additively, and thereafter, if $ \mcal C  $ is rigid, then so is $ \mcal C^\oplus $.
Furthermore, any unitary tensor functor $\mathcal{C}$ to $\mathcal{D}$ canonically extends to unitary tensor functor from $\mathcal{C}^{\oplus}$ to $\mathcal{D}$.
\begin{rem}
$ \t{proj}(\mathcal{C}^{\oplus}) $ turns out to be a unitary Karoubi envelope of the C*-category $ \mcal C $.
\end{rem}

Before proceeding further, we exhibit the connection between semi-simple, rigid C*-tensor categories and oriented factor planar algebras.
We will just highlight the main points of this correspondence which is a kind of a folklore in the C*-tensor category, quantum group and subfactor communities.
We setup the following definition which will be useful throughout this paper.
\begin{defn}
	Let $ \mcal C $ be a rigid semi-simple C*-tensor category and $\mscr X = \{X_\iota \}_\iota  $ be a family of objects in $ \mcal C $. The full subcategory of $ \mcal C $ obtained from all possible direct sums of simple objects which appear as a sub-object of a finite tensor-fold of elements from the family and their duals, will be referred as the \textit{full subcategory tensor-generated by $ \mscr X $} and denoted by $\lab \mscr X \rab $.
	When this subcategory is the whole of $ \mcal C $, we simply say \textit{$ \mscr X $ tensor-generates $ \mcal C $}.
\end{defn}

\subsection{C*-tensor category associated to an oriented planar algebra}\label{OPA2C*} $ \  $

Let $ P $ be a $ \Lambda $-oriented factor planar algebra. We define a C*-tensor category $\mcal{C}^{P}$ as follows:
\begin{itemize}
\item Objects are words $w\in W$. 
\item For two objects $ v,w $, the morphism space $\mcal{C}^{P}(v,w):=P_{w{v^*}}$.
\item For objects $ u,v,w $ and morphisms $ f \in \mcal {C}^{P}(v,w), g \in \mcal {C}^{P}(u,v) $, composition is given by the action of multiplication tangle.
That is, $ f\circ g\coloneqq P_{M^v_{w,u}}(g,f) $.
\item $*$-structure on the category is given by the $ * $-structure of the planar algebra $ P $.
\item Tensor product of objects is just the concatenation of the words and for two morphism $ f \in \mcal {C}^{P}(u,v), f' \in \mcal {C}^{P}(u',v')  $, the tensor product $ f\otimes g $ is given by the action of the tangle
\[\psfrag{a}{$ \bigstar $}
\psfrag{b}{$ v $}
\psfrag{c}{$ u $}
\psfrag{d}{$ u' $}
\psfrag{e}{$ \bigstar $}
\psfrag{f}{$ \bigstar $}
\psfrag{g}{$ f $}
\psfrag{h}{$ v'$}
\psfrag{i}{$ g $}
\raisebox{-1.5em}{\includegraphics[scale=0.3]{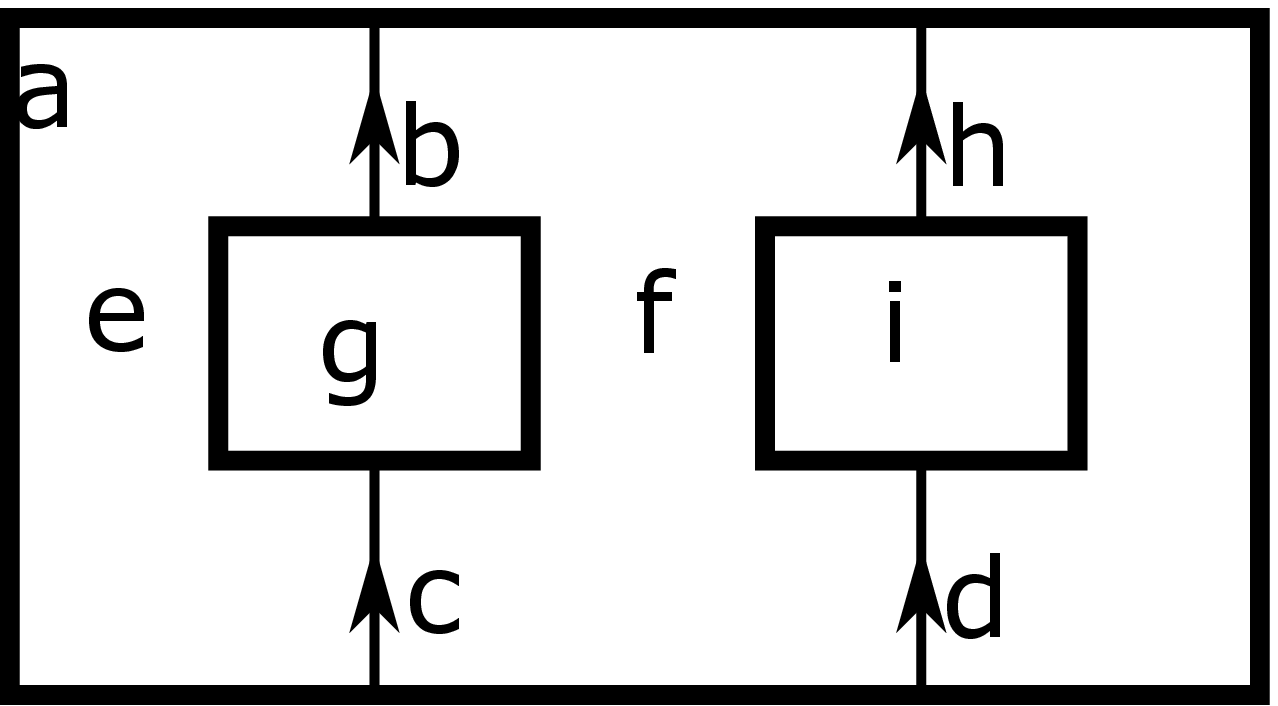}}.\]
\item The unit object is given by the empty word.
\item For each word $w\in W$, a dual object is given by ${w^*}$, with evaluation and coevaluation given by the action of tangles \psfrag{a}{$ \bigstar $}
\psfrag{b}{$ w $}
\raisebox{-1.2em}{\includegraphics[scale=0.2]{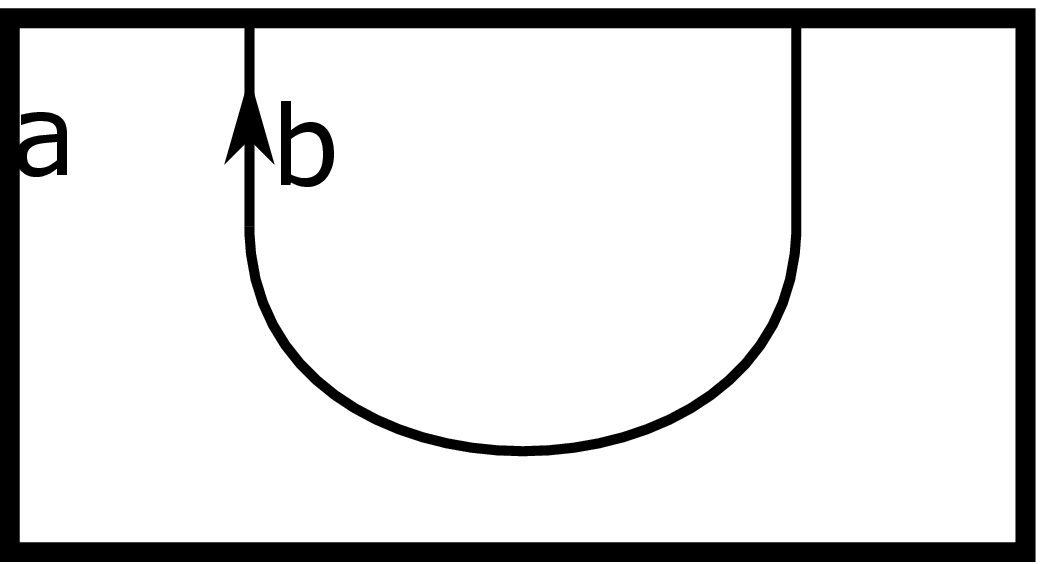}} and 
\psfrag{a}{$ \bigstar $}
\psfrag{b}{$ w $}
\raisebox{-1.2em}{\includegraphics[scale=0.2]{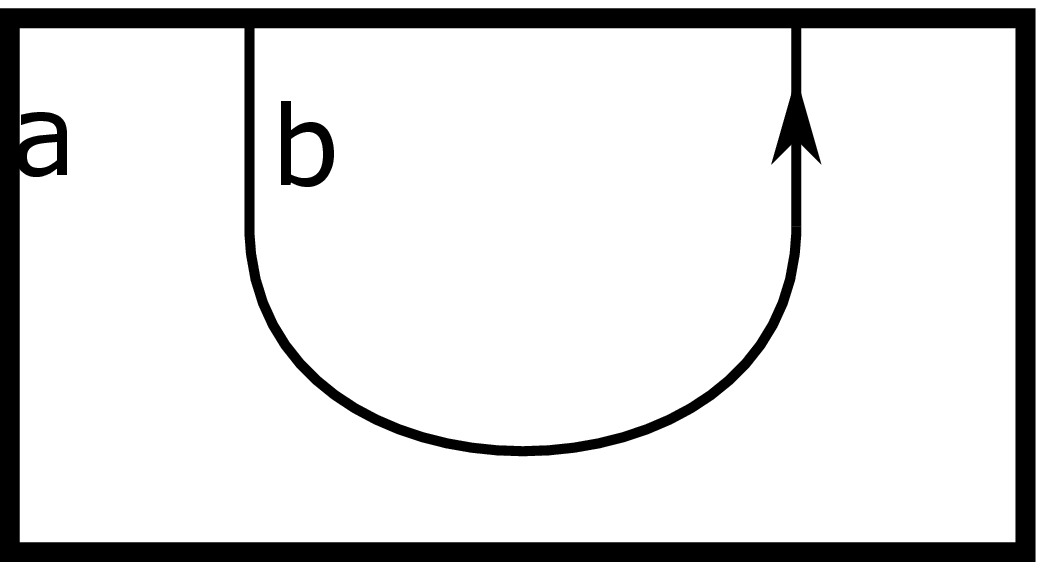}}.

\end{itemize}

While this is a C*-tensor category with duals, it does not have direct sums and subobjects.
The unitary Karaubi envelope will allow us to remedy this problem and give us a semisimple C*-tensor category.
\begin{defn} For a $\Lambda$-oriented factor planar algebra $P$, its \textit{projection category} is the rigid C*-tensor category $\mathcal{K}(\mathcal{C}^{P})$.
\end{defn}
Thus starting from a $\Lambda$-oriented factor planar algebra $P$, we get a rigid semisimple C*-tensor category $ \mathcal{K}(\mathcal{C}^{P}) $ as described above.
Note that $ \mathcal{K}(\mathcal{C}^{P}) $ is tensor-generated by $ \{1_{\lambda\oline\lambda}\}_{\lambda \in \Lambda} $.

\subsection{Oriented planar algebra associated to a C*-tensor category}\label{C*2OPA}$ \ $

Conversely, starting from a strict, rigid, semi-simple C*-tensor category $ \mcal C $ with the trivial object $ \mathbbm 1 $ being simple, and a family of objects $\mscr X = \{X_\lambda\}_{\lambda\in\Lambda} $ tensor-generating $ \mcal C $, one can define a $ \Lambda $-oriented planar algebra $ P^{\mscr X} $ (or simply $ P $) whose associated rigid C*-tensor category is equivalent to $ \mcal C $.
For this, we first fix normalized standard solutions $ (R_\lambda , \ol R_\lambda) $ to conjugate equations for each $ X_\lambda $ in the family. 
For $ \lambda \in \Lambda $, set $ X_{\ol \lambda} := \ol { X_\lambda} $, and for a word $ w := (\iota _1, \ldots , \iota _n) \in W = W_\Lambda $, let $ X_w$ denote the object $ X_{\iota _1} \otimes \cdots \otimes X_{\iota _n} $.
Define $ P_w := \mcal C (\mathbbm 1 , X_w) $ for all $ w \in W $.
Now, for a $ \Lambda $-oriented tangle $ T : (w_1 , \ldots , w_n) \ra w_0 $, one needs to define its action, that is, a multi-linear map $ P_T : P_{w_1} \times \cdots \times P_{w_n} \ra P_{w_0}$.
Let $ x_j \in P_{w_j} $ for $ 1\leq j \leq n $.
In the isotopy class of $ T $, choose a tangle diagram $ \widetilde{T} $ in \textit{standard form}, namely, (i) all discs (internal and external) are rectangles with sides parallel to the coordinate axes in $ \R^2 $, (ii) all strings are smooth with finitely many local maximas or minimas, (iii) all marked points are on the top side of every rectangle (internal or external) where the strings end transversally, and (iv) it is possible to slice $ \widetilde{T} $ into horizontal strips which contains exactly one local maxima or local minima or an internal rectangle.
For $ 1\leq j \leq n $, we label the internal rectangle in $ \widetilde{T} $ associated to the $ j^{\t {th}} $ internal disc in $ T $, by $ x_j \in P_{w_j} = \mcal C (\mathbbm 1 , X_{w_j}) $.
To each horizonatal strip of $ \widetilde{T} $, we assign a morphism prescribed by

\begin{align*}
\raisebox{-0.75 em}
{\psfrag{l}{$\lambda$}
	\psfrag{u}{$u$}
	\psfrag{v}{$v$}
	\includegraphics[scale=0.2]{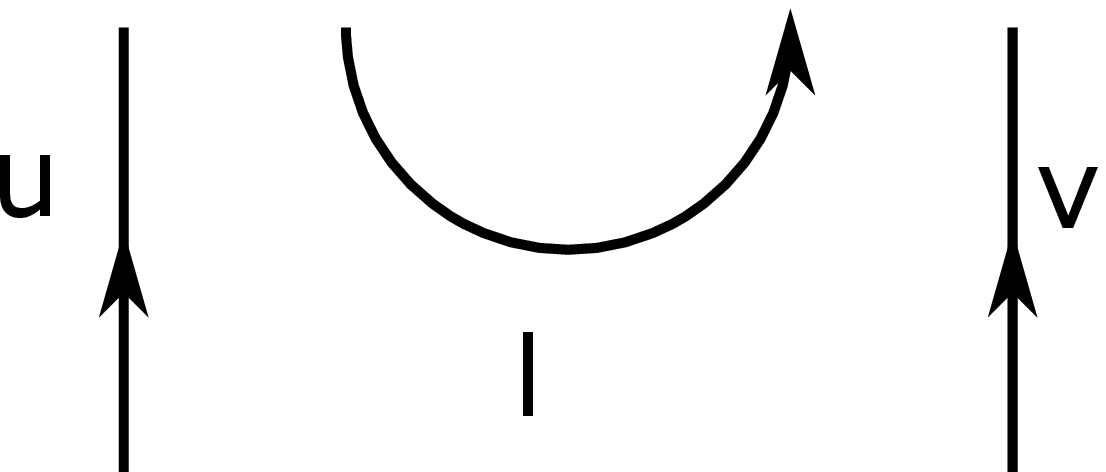}
} \longmapsto 1_{X_u} \otimes  R_{\lambda} \otimes 1_{X_v} & &
\raisebox{-0.75 em}
{\psfrag{l}{$\lambda$}
	\psfrag{u}{$u$}
	\psfrag{v}{$v$}
	\includegraphics[scale=0.2]{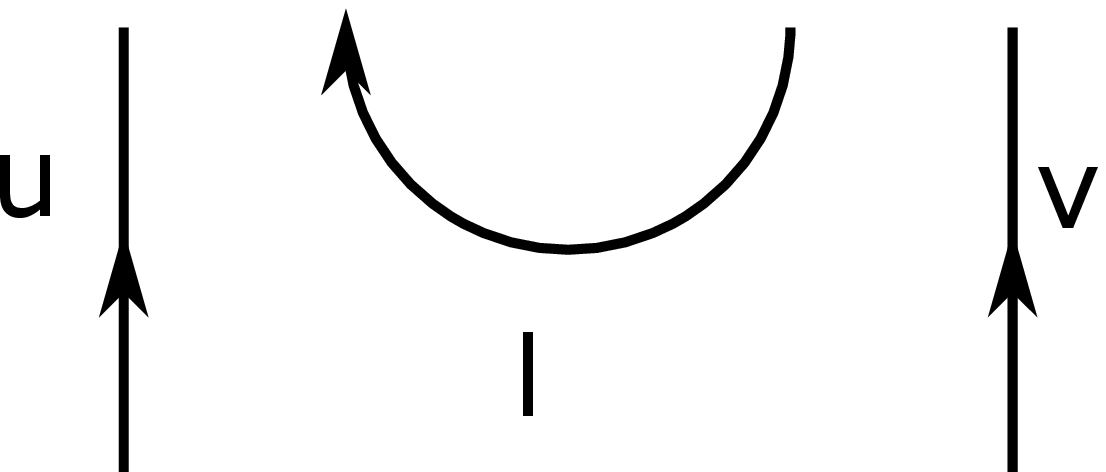}
} \longmapsto 1_{X_u} \otimes \ol R_{\lambda} \otimes 1_{X_v}\\
&&\\
\raisebox{-0.75 em}
{\psfrag{l}{$\lambda$}
	\psfrag{u}{$u$}
	\psfrag{v}{$v$}
	\includegraphics[scale=0.2]{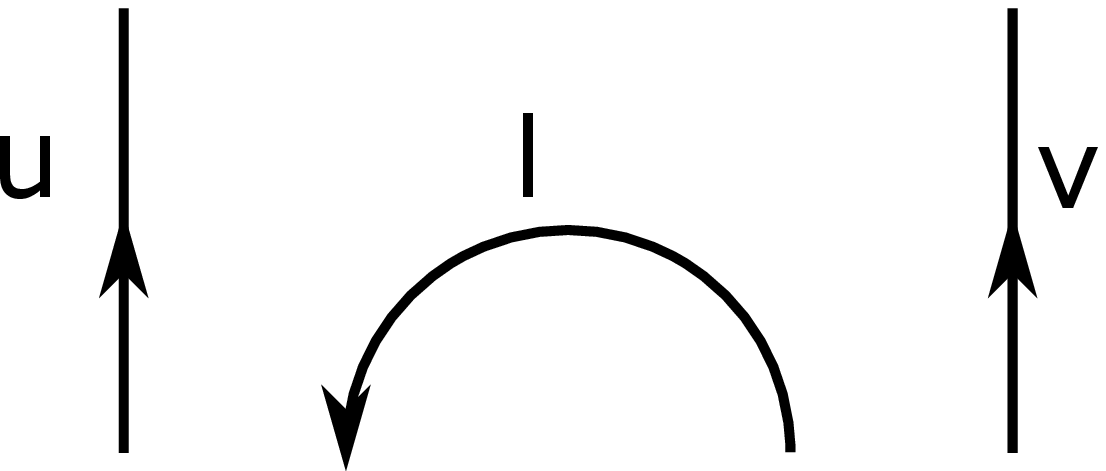}
} \longmapsto 1_{X_u} \otimes  R^*_{\lambda} \otimes 1_{X_v} & &
\raisebox{-0.75 em}
{\psfrag{l}{$\lambda$}
	\psfrag{u}{$u$}
	\psfrag{v}{$v$}
	\includegraphics[scale=0.2]{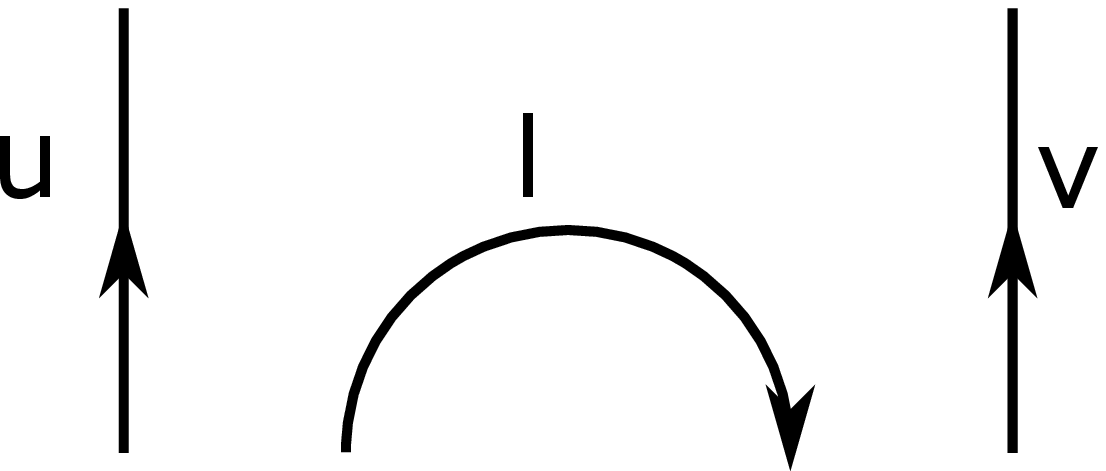}
} \longmapsto 1_{X_u} \otimes {\ol R}^{\; *}_{\lambda} \otimes 1_{X_v}
\end{align*}
\[
\raisebox{-0.75 em}
{\psfrag{l}{$\lambda$}
\psfrag{u}{$u$}
\psfrag{v}{$v$}
\psfrag{w}{$w_j$}
\psfrag{x}{$x_j$}
\includegraphics[scale=0.2]{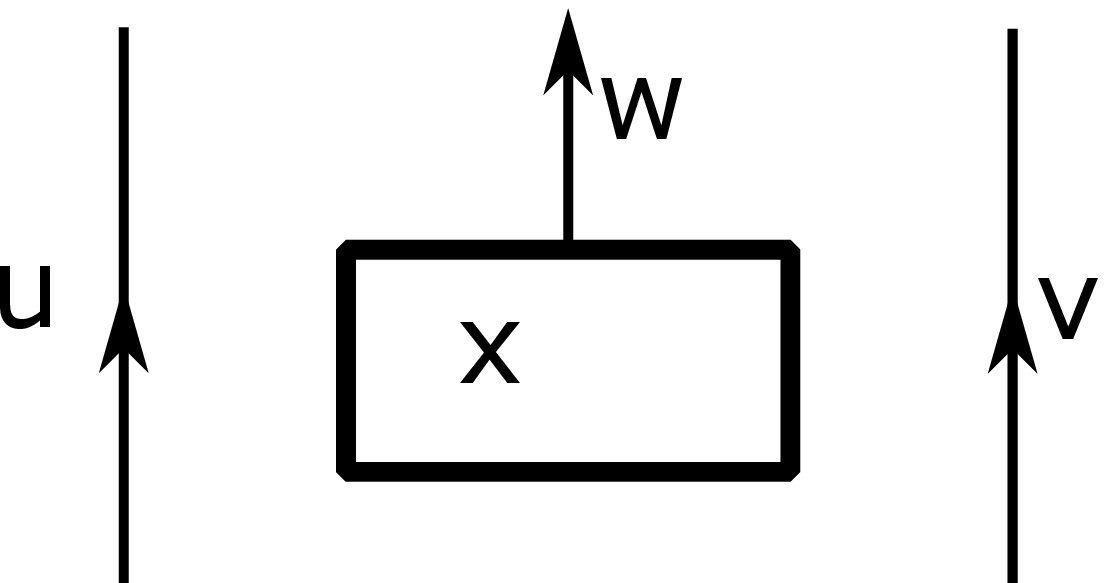}
} \longmapsto 1_{X_u} \otimes  x_j \otimes 1_{X_v}
\]
Define $ P_T(x_1,\ldots , x_n) $ as the composition of the morphisms associated to each horizontal strip.
The action is indeed well defined and the planar algebra is a factor planar algebra.
Observe that we use the strictness of the tensor structure in $ \mcal C $ at every step of this construction, starting from the definition of $ X_w $ till the action of tangles and its invariance under isotopy.
One can possibly extend this construction in the non-strict case as well by introducing appropriate associators.
\begin{rem}\label{OPA-C*-correspondence}
We summarize the above discussion.
Any rigid semi-simple C*-tensor category tensor-generated by $ \mscr X $ is equivalent to the $\mathcal{K}(\mathcal{C}^{P^{\mscr X}}) $.
Conversely, the projection category of any $ \Lambda $-oriented factor planar algebra $ P $, is tensor-generated by $ \mscr X \coloneqq \{1_{\lambda \ol \lambda}\}_{\lambda \in \Lambda} $; further, $ P $ is isomorphic to the $ \Lambda $-oriented factor planar algebra $ P^{\mscr X} $.
\end{rem}
\begin{rem}
Given any $ \Lambda $-oriented factor planar algebra $ P $, there exists a $ \rm{II}_1 $-factor $ N $ and a family $ \{X_\lambda\}_{\lambda \in \Lambda} $ of extremal, bifinite $ N $-$ N $-bimodules such that the $ \Lambda $-oriented planar algebra associated to this family is isomorphic to $ P $.
To see this, one considers the projection category $ \mathcal{K}(\mcal C^P )$ associated to $ P $.
Now by \cite{BHP}, any rigid, semi-simple C*-tensor category, in particular $ \mcal C^P $, is equivalent to a full subcategory of extremal, bifinite bimodules over a $ \rm{II}_1 $-factor.
Applying the converse above, one gets the required result.
\end{rem}
\comments{As an example, one can start with a $ II_1 $-factor $ N $ and a family $ \{X_\lambda\}_{\lambda \in \Lambda} $ of extremal, bi-finite $ N $-$ N $-bimodules

Let $ P $ be a $ \Lambda $-oriented factor planar algebra and $ \mcal C $ be its corresponding projection category.
Set $ \mcal L:= \left\{1_\lambda\oplus 1_{\bar{\lambda}}: \lambda \in \Lambda\right\} $ and $ \tilde{P} $ be the $ \mcal L $-oriented factor planar algebra obtained by above prescription.
Note that, $ \tilde{P} $ is a factor planar algebra in the sense of \cite{BHP}.
Let $ \tilde{\mcal C} $ be its projection category.
Then we have the following
\begin{prop}\label{proj}
$ \mcal C \cong \tilde{\mcal C} $ as tensor categories.
\end{prop}
\begin{proof}
\red{PENDING}	
\end{proof}
\begin{rem}
Thus by \Cref{proj} and \cite{BHP}, there is a $ II_1 $-factor $ M $ such that, $ \mcal C $ is equivalent to a full subcategory of extremal bifinite bimodules over M.
Further, if $ \Lambda $ is countable then $ M $ can be chosen to be $ L(\F_\infty) $.
\end{rem}

\vskip 1em}
\subsection{Free product of categories}\label{free prod prel} $ \ $

The notion of free product of two semi-simple $ \t C^* $-tensor categories with simple tensor units is due to Bisch and Jones.
In this article, we will use the version which was elaborated in \cite{GJR}.
We breifly include the notational set up and a fundamental result on free products here for the convenience of the reader.

Let $ \mcal C_+ $ and $ \mcal C_- $ be two strict, semi-simple $\t C^* $-tensor categories with simple tensor units $ \mathbbm {1}_+ $ and $ \mathbbm {1}_- $ respectively.
In our construction, we pick a strict model of $\mcal{C}_{\pm}$.
Let $ \Sigma $ be the set of words with letters in $ \t {Obj}(\mcal C_+) \cup \t {Obj}(\mcal C_-)$.
To a word $ \sigma \in \Sigma$, we associate the sub-word $ \sigma_+ \in \t {Obj}(\mcal C_+)$ (resp., $ \sigma_- \in \t {Obj}(\mcal C_-) $) by dropping all the letters in $ \sigma $ coming from $ \t {Obj}(\mcal C_-) $ (resp., $\t {Obj}(\mcal C_+) $).
The object obtained by tensoring the letters in $ \sigma_\pm $ will be denoted by $ t( \sigma_\pm )$ with the convention $ t(\varnothing) = \mathbbm {1_\pm}$ where appropriate.
For instance, if $ \sigma = a^+_1 a^-_2 a^+_3 a^-_4 a^-_5 $, then $ \sigma_+ =  a^+_1 a^+_3$, $t (\sigma_+) = a^+_1 \otimes a^+_3$, $ \sigma_- = a^-_2 a^-_4 a^-_5 $ and $t(\sigma_-) = a^-_2 \otimes a^-_4 \otimes a^-_5 $.
\begin{defn}
	Let $ \sigma, \tau \in \Sigma $.
	A `\textit{$ (\sigma,\tau )$-NCP}' consists of:
	\begin{itemize}
		\item  a \textit{non-crossing partitioning} of the letters in $ \sigma $ and $ \tau $ arranged at the bottom and on the top edges of a rectangle respectively moving from left to right, such that each partition block consists only of objects from $ \mcal C_+ $ or only of objects from $ \mcal C_- $ ,
		\item every block gives a pair of sub-words (possibly empty) of $ \sigma $ and $ \tau $, say, $ (\sigma_0,\tau_0) $, where $ \sigma_0 $ (resp. $ \tau_0 $) consists of letters in the partition coming from $ \sigma $ (resp. $ \tau $).
		For each such block, seen as a rectangle with the bottom labelled by $\sigma_0$ and the top labelled by $ \tau_0  $, we choose a morphism from $t(\sigma_0)  $ to $ t(\tau_0) $ in the appropriate category.		
	\end{itemize}
\end{defn}
We give an example of a $ (\sigma , \tau) $-NCP (from \cite{GJR}) in following figure, where $ \sigma = a^+_1 a^+_2 a^+_3 a^-_4$ $a^+_5a^+_6 a^-_7 a^+_8  $ and $ \tau =  b^+_1 b^-_2 b^-_3 b^+_4 b^+_5 $ with $ a_i^\vlon,b_{j}^\vlon \in \mcal C_\vlon$, $ \vlon \in \{+,-\} $.

\begin{center}
	\psfrag{1}{$ a^+_1 $}
	\psfrag{2}{$ a^+_2 $}
	\psfrag{3}{$ a^+_3 $}
	\psfrag{4}{$ a^-_4 $}
	\psfrag{5}{$ a^+_5 $}
	\psfrag{6}{$ a^+_6 $}
	\psfrag{7}{$ a^-_7 $}
	\psfrag{8}{$ a^+_8 $}
	\psfrag{f}{$ f_1 $}
	\psfrag{g}{$ f_2 $}
	\psfrag{h}{$ f_3 $}
	\psfrag{i}{$ f_4 $}
	\psfrag{j}{$ f_5 $}
	\psfrag{A}{$ b^+_1 $}
	\psfrag{B}{$ b^-_2 $}
	\psfrag{C}{$ b^-_3 $}
	\psfrag{D}{$ b^+_4 $}
	\psfrag{E}{$ b^+_5 $}
	\includegraphics[scale=0.3]{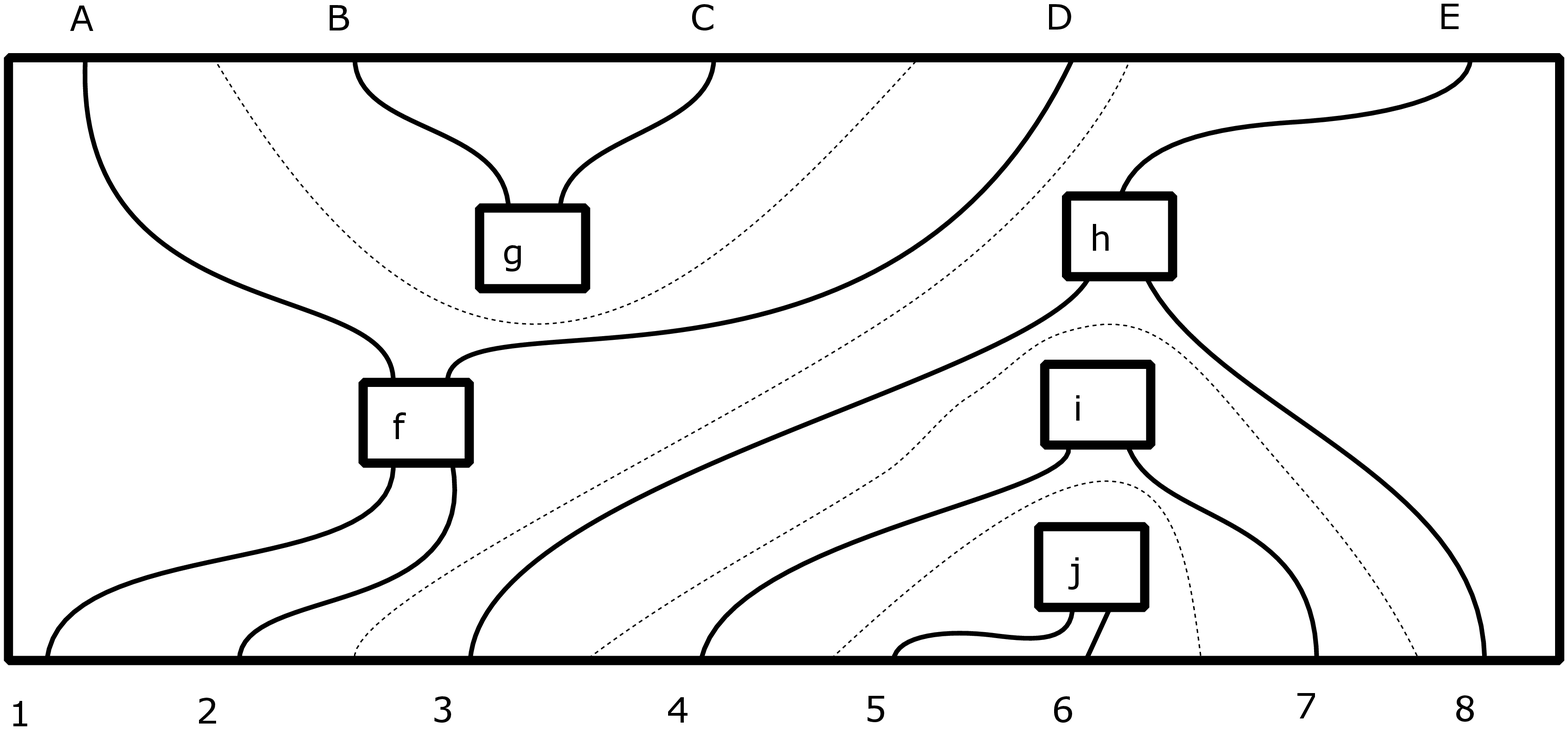}
\end{center}
Here, the pair of subwords corresponding to the partition blocks are $ \rho_1 = (a_1^+a_2^+,b_1^+b_4^+), \rho_2=(\varnothing,b_2^-b_3^-),\rho_3=(a_3^+a_8^+,b_5^+),\rho_4=(a_4^-a_7^-,\varnothing),\t{and }\rho_5= (a_5^+a_6^+,\varnothing) $.
Note that each letter of $ \rho_i$ either belongs Obj$(\mcal C_+) $ alone or Obj$(\mcal C_-) $ alone, for every $ i $ and each of $ \rho_i $ is assigned a morphism from the corresponding category.
For instance, all the letters of $ \rho_3 $ are objects of $ \mcal C_+ $ and is assigned the morphism $ f_3 \in \mcal C_+(a_3^+\otimes a_8^+,b_5^+) $.

We denote the set of all $ (\sigma ,\tau) $-NCPs by $ NCP(\sigma,\tau) $.
Now, to every $ T \in NCP(\sigma , \tau ) $, we can associate unique $ T_{\pm}\in NCP(\sigma_{\pm} , \tau_{\pm} )$  by deleting all blocks whose letters are labeled by the opposite sign.

Since all letters in $ \sigma_{\pm} $ and $ \tau_{\pm} $ come from either $ \mcal C_+ $ or $ \mcal C_- $ only, the non-crossing partitions  $T_{\pm} $ give rise to unique morphisms $ Z_{T_\pm}\in \mcal C_\pm \left(t(\sigma_\pm) , t(\tau_\pm) \right) $ using the standard graphical calculus for monoidal categories.

So, for any $ \sigma ,\tau \in \Sigma $ and $ T \in NCP (\sigma  , \tau) $, we have morphisms $ Z_{T_\pm} \in \mcal C_\pm \left(t(\sigma_\pm) , t(\tau_\pm) \right) $.
We write $ Z_T\coloneqq Z_{T_+} \otimes Z_{T_-} \in \mcal C_+ \left(t(\sigma_+), t(\tau_+)\right) \otimes \mcal C_- \left(t(\sigma_-), t(\tau_-)\right)$.
For example, for the NCP $ T $ in above figure, 
$$ Z_{T_+} = (f_1\otimes f_3)\circ (1_{a_1^+\otimes a_2^+\otimes a_3^+}\otimes f_5 \otimes1_{a_8^+})\ \text{and}\ Z_{T_+} = f_2\circ f_4$$

We can now describe the category $ \mcal{NCP} $ as follows.
\begin{itemize}
	\item Objects in $\mcal{NCP}$ are given by $\Sigma$.
	
	\bigskip
	
	\item For $ \sigma , \tau \in \Sigma $, the hom space between them is given by
	$$\mcal {NCP} (\sigma , \tau) \coloneqq \t {span} \left\{Z_T: T \in NCP(\sigma , \tau) \right\} \subset \mcal C_+ \left(t(\sigma_+), t(\tau_+)\right) \otimes \mcal C_- \left(t(\sigma_-), t(\tau_-)\right).$$
	
\end{itemize}

Composition of morphisms is given by composing the tensor components, which is obviously bilinear, and associative. There is also a $ * $-structure given by applying $ * $ on each of the tensor components. The tensor structure arises in the obvious way, by horizontal juxtaposition of non-crossing partitions, and taking the tensor product of the associated morhisms. This is a C*-tensor category (see \cite{GJR} for detailed explanations and proofs of these statements).
Further, for $ \vlon = \pm $, we have an obvious fully faithful unitary tensor functor $\gamma_{\vlon} : \mcal C_\vlon \lra \mcal {NCP} $ sending $ X \in \t {Obj} (\mcal C_\vlon) $ to the length one word $X \in \t {Obj} (\mcal {NCP})$.
The tensor preserving property of $ \gamma_\vlon $ is implemented by $ J^{\gamma_\vlon}_{X,Y} \coloneqq 1_{X\otimes Y} \otimes 1_{\mathbbm 1_-} $ or $ 1_{\mathbbm 1_+} \otimes 1_{X\otimes Y}$ in $ \mcal {NCP} \left(\gamma_\vlon (X) \otimes \gamma_\vlon (Y) , \gamma_\vlon (X\otimes Y )\right) $ according as $ \vlon $ is $ + $ or $ - $.
\begin{defn}\label{elemncp}
A $(\sigma , \tau) $-NCP $T$ will be called \textit{elementary} if $ \sigma = (\sigma_1, \ldots , \sigma_n) $,  $ \tau = (\tau_1, \ldots,\tau_n) $ for some $ n $, and the only block partitions of $ T $ are $ (\sigma_i , \tau_i) $ for $ 1 \leq i \leq n $ where at most one of $ \sigma_1, \ldots , \sigma_n, \tau_1, \ldots,\tau_n $ is empty.
\end{defn}\begin{rem}\label{elemncpgen}
Any morphism in $ \mcal {NCP} $ can be expressed as a linear combination of composition of $ Z_T $'s for $ T $ being elementary NCP.
\end{rem}



\begin{defn}[{\cite{GJR}}]
	The \textit{free product of the categories $ \mcal C_+ $ and $ \mcal C_- $} (as above) is defined as the unitary idempotent completion proj$(\mcal{NCP})$ which we denote by $\mathcal{C}_{+}*\mathcal{C}_{-}$. 
\end{defn}

From \Cref{projcat}, it is clear that $\mathcal{C}_{+}*\mathcal{C}_{-}$ is a C*-tensor category with simple tensor unit containing $ \mcal C_\pm $ as full tensor subcategories via the fully faithful unitary tensor functors
\[
\iota_\pm \ : \ \mcal C_\pm \ \os {\displaystyle \gamma_\pm} \lra \ \mcal {NCP} \ \os {\displaystyle \alpha^{\mcal {NCP}}} \lra \ \mathcal{C}_{+}*\mathcal{C}_{-}\ .
\]

\begin{defn}
	Let $\Irr(\mcal{C}_{\pm})$ denote a choice of object from each isomorphism class of simple objects, such that the tensor units are chosen to represent their isomorphism class.
	Then $\Sigma_{0}:= \{\varnothing\} \cup \left\{a^{\vlon_{1}}_{1} \ldots  a^{\vlon_{k}}_{k} \; : \;  k\in \N, \ \vlon_{i} \in \{ \pm \} , \ a^{\vlon_i}_{i}\in \Irr(\mcal{C}_{\vlon_i})\setminus\{\mathbbm{1}_{\vlon_i}\},\ \vlon_{i} = -\vlon_{i+1} \t { for } 1\leq i \leq k \right\} $.
\end{defn}
With the above notations, we have the following proposition (see \cite{GJR} for proof).
\begin{prop}\label{freeprod}
	The free product $ \mcal C_+ \ast \mcal C_-  $ is a strict, semi-simple C*-tensor category with $ \mcal C_\pm $ as full tensor subcategories (via $ \iota_\pm $) which tensor-generate $ \mcal C_+ \ast \mcal C_- $.
	Moreover, the fully faithful unitary tensor functor $ \alpha^{\mcal {NCP}} : \mcal{NCP} \ra \mcal C_+ \ast \mcal C_- $  gives rise to a bijection between $\Sigma_{0}$ and isomorphism classes of simple objects in $\mcal{C}_{+}*\mcal{C}_{-}$.
	Further, if $ \mcal C_\pm  $ are rigid, then so is $ \mcal C_+ \ast \mcal C_- $.
\end{prop}

\begin{rem}
	By \Cref{basics-cat} and above proposition, $ \mcal K(\mcal {NCP}) \cong  \mcal C_+ \ast \mcal C_- $.
\end{rem}

The construction of the free product in \cite{GJR} is explicit, but was tailored to graphical calculus considerations. There are many other possible candidates for a free product construction. To justify calling it \textit{the} free product, we need to verify that it satisfies a universal property.
\begin{thm}\label{freeprodunilem}
	Let $ \mcal C $, $ \mcal C_+ $ and $ \mcal C_- $ be a strict, rigid, semi-simple C*-tensor categories and $ F_\pm : \mcal C_\pm \ra \mcal C $ be unitary tensor functors.
	Then, there exists a triplet $\left( \widetilde{F} \ ,\ \kappa_+ \ ,\ \kappa_- \right)$ where $ \widetilde F : \mcal C_+ \ast \mcal C_- \ra \mcal C $ is a unitary tensor functor and $ \kappa_\pm : F_\pm \ra \widetilde F \circ \iota_\pm $ are unitary monoidal natural isomorphisms.
	Moreover, $\left( \widetilde{F} \ , \ \kappa_+ \ , \ \kappa_- \right)$ is unique up to a unique unitary monoidal natural isomorphism compatible with $ \kappa_\pm $.
\end{thm}
\begin{proof}
	We will first construct a unitary tensor functor $ G: \mcal{NCP}\ra \mcal C $ such that $ F_\pm = G \circ \gamma_\pm $.
	Suppose such a $ G $ exists.
	Since $ \mcal C $ is a semi-simple, strict C*-tensor category, by \Cref{projcat}(1) and (2), $ \alpha^{\mcal C} :\mcal C \ra \t{proj}(\mcal C)$ is a monoidal C*-equivalence.
	Choose $ \beta: \t{proj}(\mcal C) \ra \mcal C $ such that $ \alpha^{\mcal C} \circ \beta $ and $ \beta \circ \alpha^{\mcal C} $ are monoidally equivalent to the corresponding identity functors via natural unitaries.
	In particular, let $ \lambda: \t {id}_{\mcal C} \ra \beta \circ \alpha^{\mcal C}$ be natural monoidal unitary.
	Then, $ \kappa_\pm \coloneqq \lambda_{F_\pm} : F_\pm \ra \beta \circ \alpha^{\mcal C} \circ F_\vlon$ is also a natural monoidal unitary.
	Define $ \widetilde F \coloneqq \beta \circ \t {proj}(G)\  : \ \mcal C_+ \ast \mcal C_- \ra \mcal C$ which is a unitary tensor functor by \Cref{projcat}(3).
	For $ \vlon =\pm $, the restriction of $ \widetilde F $ to $ \mcal C_\vlon $ is given by
	\[
	\widetilde F \circ \iota_\vlon \
	= \ \beta \circ \t {proj}(G) \circ \alpha^{\mcal {NCP}} \circ \gamma_\vlon = \beta \circ \alpha^{\mcal C} \circ G \circ \gamma_\vlon
	\ =\ \beta \circ \alpha^{\mcal C} \circ F_\vlon
	\ \os {\displaystyle \kappa_\vlon}\longleftarrow
	\ F_\vlon
	\]
	where the second last equality comes from \Cref{projcat}(3).
	\vskip 1em
	We will now construct $ G $.
	While applying the functor $ F_\pm $ on an object or a morphism, we will often drop the sign in the suffix and simply write $ F $; the sign can be read off from the category $ \mcal C_\pm $ in which the object or the morphism belongs.
	For $ \sigma = X_1\ \ldots\ X_m \in \t {Obj} (\mcal {NCP}) $, define 
	\begin{align*}
		G\sigma \coloneqq & \  F X_1 \otimes \cdots \otimes F X_m \in \t {Obj}( \mcal C) \t {, and}\\
		F \sigma \coloneqq & \ \t { the word } F (X_1) \ \ldots \ F (X_m) \ .
	\end{align*}
	
	To define $ G$ at the level of morphisms, consider $ \sigma, \tau \in  \t {Obj} (\mcal {NCP}) $ and a $ (\sigma,\tau) $-NCP $ T $.
	For $ \vlon = \pm$, suppose $ J^\vlon : \otimes \circ (F_\vlon \times F_\vlon) \ra F_\vlon \circ \otimes  $ (resp., $ \eta_\vlon : F_\vlon \mathds 1_\vlon \ra \mathds 1 $) is a natural unitary (resp., unitary) which implements the tensor- (resp., unit-) preserving property of $ F_\vlon $.
	Again, applying $ F$ on the morphisms assigned to the partition blocks of $ T $ and composing with appropriate $ J^\vlon $'s and $ \eta_\vlon $'s, we get a $ (F\sigma , F\tau) $-NCP which we denote by $ FT $.
	Note that the letters of $ F\sigma $ and $ F\tau $, and the morphisms assigned to the partition blocks of $ FT $ all belong to the strict tensor category $ \mcal C $.
	Using graphical calculus of morphisms in $ \mcal C$, the NCP $ FT $ yield a unique morphism $Z^{\mcal C}_{FT}  \in \mcal C (G\sigma , G\tau )$.
	The obvious choice of $ G $ would be
	\[
	\mcal{NCP}(\sigma , \tau) \in Z_T = Z_{T_+} \otimes Z_{T_-} \os {\displaystyle G} \longmapsto Z^{\mcal C}_{FT} \in \mcal C (G\sigma , G \tau)
	\]
	and extending it linearly.
	
	We need to show that $ G $ is well-defined at the level of morphisms.
	Consider two objects $ \sigma = X_1\ \ldots\ X_m$ and $\tau $ in  $\mcal {NCP}$ and two $ (\sigma ,\tau) $-NCPs $ S $ and $ T $.
	For $ 1 \leq i \leq m $, choose a dual $\ol X_i $ of $ X_i $, and a normalized standard solution $ (R_i , \ol R_i) $ to the conjugate equation for the duality of $ (X_i , \ol X_i) $.
	Generate the positive, faithful `left' trace $\t {tr}_{t(\sigma_\pm)}$ on the endomorphism space $ \mcal C_\pm \left(t(\sigma_\pm),t(\sigma_\pm)\right) $ using the $ R_i$'s.
	We then have the following inner product on the space $ \mcal {NCP} (\sigma , \tau) $:
	\[
	\lab Z_S , Z_T \rab \coloneqq \t {tr}_{t(\sigma_+)} \otimes \t {tr}_{t(\sigma_-)} \left( Z_{T^* \circ S} \right) = \t {tr}_{t(\sigma_+)} \left( Z_{(T^*\circ S)_+} \right) \ \ \t {tr}_{t(\sigma_-)} \left( Z_{(T^* \circ S)_-} \right).
	\]
	Suppose $ X_i $ lies in $ \mcal C_{\vlon_i} $ for $ 1 \leq i \leq m $.
	Define $ R'_i \coloneqq \left( J^{\vlon_i} _{ \ol X_i , X_i }\right)^* \circ F(R_i ) \circ \eta^*_{\vlon_i}$ and $ \ol R'_i \coloneqq \left( J^{\vlon_i} _{X_i , \ol X_i }\right)^* \circ F(\ol R_i ) \circ \eta^*_{\vlon_i}$.
	One can easily check (using the equations satisfied by $ J^\vlon $ and $ \eta_\vlon $ in making $ F_\vlon $ a monoidal functor) that $ (R'_i , \ol R'_i) $ is a solution to the conjugate equation (possibly not standard) for the duality of $ (FX_i , F \ol X_i) $.
	Using these solutions in an obvious way, we generate the solution $ (R' , \ol R') $ to the conjugate equation for $ \left(G \sigma , F \ol X_m \otimes \cdots \otimes F \ol X_1 \right) $; for instance, $ R' \coloneqq \left(1^{\ }_{F \ol X_m \otimes \cdots \otimes F \ol X_2} \otimes R'_1 \otimes 1^{ \ }_{FX_2 \otimes \cdots \otimes FX_m}  \right) \cdots \left(1_{F \ol X_m} \otimes R'_{m-1} \otimes 1_{FX_m}  \right) R'_m$.
	Consider the positive, faithful (possibly not tracial) functional $ \vphi^{\ }_{G\sigma} \coloneqq {R'}^* \left(1_{F \ol X_m \otimes \cdots \otimes F \ol X_1} \otimes \bullet\right) R'$ on the endomorphism space $ \mcal C (G\sigma , G\sigma) $.
	A careful observation and some straight-forward calculations will tell us
	\[
	\lab Z_S , Z_T \rab = \vphi^{\ }_{G \sigma} \left( Z^{\mcal C}_{FT^* \circ FS} \right) = \vphi^{\ }_{G \sigma} \left( \left(Z^{\mcal C}_{FT} \right)^* \circ Z^{\mcal C}_{FS} \right)  = \vphi^{\ }_{G \sigma} \left( \left( G(Z_T) \right)^* \circ G (Z_S) \right)
	\]
	where the first equality can be derived by crucially using the fact that any NCP $ T' $ can be expressed as a non-crossing overlay of $ T'_+ $ and $ T'_- $.
	Thus, $ G : \mcal{NCP} (\sigma , \tau) \ra \mcal C(G \sigma , G \tau) $ is well-defined and injective as well.
	
	The remaining properties for $ G $ being a unitary tensor functor is routine to verify, and so is the condition $ F_\pm = G \circ \gamma_\pm  $.
	\vskip 1em
	We are now left to establish the uniqueness part.
	Let $ \left( H^i , \kappa^i_+ , \kappa^i_- \right)$ for $ i=1,2 $ be two triplets satisfying the conditions in the statement of this theorem.
	We need to find a unitary natural monoidal isomorphism $ \mu :H^1 \ra H^2 $ such that $ \mu_{\iota_\vlon} \circ \kappa^1_\vlon = \kappa^2_\vlon $ for $ \vlon = \pm $, and show that such a $ \mu $ is unique.
	
	The compatibility condition forces us to set
	\[
	\mu_{\iota_\vlon (X)} \ \coloneqq \  \kappa^2_{\vlon ,X} \left(\kappa^1_{\vlon ,X}\right)^* \ \in \ \mcal C \left(H^1 (\iota_\vlon (X)) \ , \ H^2 (\iota_\vlon (X)) \right) \t{ for } X \in \t {Obj} (\mcal C_\vlon) .
	\]
	
	Next, we intend to define $ \mu $ one level higher, namely, for objects in $ \mcal {NCP} $.
	Consider $ \sigma = (X_1 , \ldots , X_m) \in \t {Obj} (\mcal {NCP}) $ for $ X_j \in \t {Obj} \mcal (C_{\vlon_j}) $.
	For this we set up the following convention which will come handy in the rest of the proof.
	\vskip 1em
	\noindent\textbf{Notation:} Let $ A: \mcal D \ra \mcal E $ be a tensor functor between strict tensor categories where $ J : \otimes  \circ (A\times A) \ra A \circ \otimes $ is the natural transformation implementing the tensor preserving property of $ A $.
	For a nonempty word $\sigma =  X_1 \ \ldots \ X_m $ with letters in $ \t {Obj} (\mcal D) $, using the $ J_{\bullet , \bullet}  $'s iteratively, we may obtain a morphism in $ \mcal E $
	\[
	J_\sigma \ : \ A (X_1) \otimes \cdots \otimes A(X_m) \ \lra \ A( X_1 \otimes \cdots \otimes X_m)
	\]
	which is independent of any iterative algorithm by the commuting hexagonal diagram (in fact, a square due to strictness of $\mcal D$ and $ \mcal E $) satisfied by $ J_{\bullet , \bullet} $.
	If $ \sigma $ has length one, then set $ J_\sigma \coloneqq 1_{A(X_1)} $.
	Note that $ J_\sigma  $ is natural in the letters of $ \sigma $.
	For $ \sigma = \sigma_1 \ \ldots\ \sigma_n $ (where $ \sigma_j $'s are nonempty subwords of $ \sigma $), we have the formula:
	\begin{equation}\label{Jformula}
	J_\sigma = J_{t(\sigma_1) \ \ldots \ t(\sigma_n)} \left( J_{\sigma_1} \otimes \cdots \otimes J_{\sigma_n} \right)
	\end{equation}
	where $ t(\cdot) $ continues to denote tensoring the letters from left to right.
	\vskip 1em
	For $ \vlon = \pm $, $ i=1,2 $, let $ J^\vlon_{\bullet , \bullet} $ and $ J^i_{\bullet , \bullet} $ denote the natural unitaries implementing the tensor preserving property of $ F_\vlon $ and $ H^i $ respectively.
	As objects in $ \mcal {NCP} $, we have $ \sigma =  \gamma_{\vlon_1} (X_1) \otimes \cdots \otimes \gamma_{\vlon_m} (X_m)$.
	On both sides, applying $ \alpha^{\mcal {NCP}} :\mcal {NCP} \ra \mcal C_+ \ast \mcal C_-$ (which we denote simply by $ \alpha $ for notational convenience), we get $ \alpha (\sigma) = \iota_{\vlon_1} (X_1) \otimes \cdots \otimes \iota_{\vlon_m} (X_m)$ since
	$ \alpha $ is trivially monoidal.
	Since we are looking for a monoidal $ \mu : H^1 \ra H^2 $, the only potential candidate must be defined as:
	\begin{align*}
		\mu_{\alpha(\sigma)} \coloneqq J^2_{\iota_{\vlon_1} (X_1) \ \ldots \ \iota_{\vlon_m} (X_m)} \left(\mu_{\iota_{\vlon_1} (X_1)} \otimes \cdots \otimes \mu_{\iota_{\vlon_m} (X_m)} \right)& \left(J^1_{\iota_{\vlon_1} (X_1) \ \ldots \ \iota_{\vlon_m} (X_m)}\right)^*\\
		& \in \mcal C \left( H^1 (\alpha (\sigma)) , H^2 (\alpha (\sigma)) \right) 
	\end{align*}
	for $ \sigma \in \t {Obj} (\mcal {NCP}) $.
	Clearly, if length of $ \sigma $ is $ 1 $, $ \mu_{\alpha(\sigma)} $ matches with the one defined before, that is, $ \mu_{\alpha \gamma_\vlon (\bullet)} = \mu_{\iota_\vlon (\bullet)}  $.
	
	It is clear (using \Cref{Jformula}) from the definition that $ \mu_\alpha : H^1 \circ \alpha \ra H^2 \circ \alpha $ is a monoidal unitary but we still need to check naturality of $ \mu_\alpha $, that is, for $ T \in NCP (\sigma ,\tau) $, we need to prove $ \mu_{\alpha(\tau)} \circ H^1 (\alpha (Z_T)) = H^2 (\alpha (Z_T)) \circ \mu_{\alpha(\sigma)} $.
	By \Cref{elemncpgen}, it is enough to obtain $ \mu_{\alpha(\tau)} \circ H^1 (\alpha(Z_T)) = H^2 (\alpha (Z_T)) \circ \mu_{\alpha(\sigma)} $ for elementary $ T \in NCP(\sigma , \tau) $.
	
	Let $ T $ be as in \Cref{elemncp} where the block partition $ (\sigma_j , \tau_j) $ is labelled by $ f_j \in \mcal C_{\vlon_j} (t(\sigma_j) , t(\tau_j)) $ for $ 1 \leq j \leq n  $.
	Let $ T_j $ denote the $ (\sigma_j , \tau_j) $-NCP with a single block partition labeled with $ f_j $.
	Then, $ Z_T = Z_{T_1} \otimes \cdots \otimes Z_{T_n} $ implying $ \alpha(Z_T) = \alpha(Z_{T_1}) \otimes \cdots \otimes \alpha(Z_{T_n}) $.
	Suppose $ \sigma_j = (X^j_1, \ldots , X^j_{k_j}) $ and $ \tau_j = (Y^j_1, \ldots , Y^j_{l_j}) $ for $ 1 \leq j \leq n $; here we are assuming that all $ \sigma_j $'s and $ \tau_j $'s are nonempty.
	Thus,
	\[
	H^1(\alpha (Z_T)) = J^1_{\alpha (\tau_1)  \ \ldots \  \alpha (\tau_n)} \left[H^1(\alpha (Z_{T_1})) \otimes\cdots \otimes H^1 (\alpha (Z_{T_n}))\right] \left( J^1_{\alpha (\sigma_1)  \ \ldots \   \alpha (\sigma_n)} \right)^* \ .
	\]
	Using \Cref{Jformula}, $ \mu_{\alpha(\tau)} \circ H^1 \alpha(Z_T)$ can be expressed as
	\begin{equation}\label{munat1}
	\begin{split}
	J^2_{\iota_{\vlon_1} (Y^1_1) \ \ldots \ \iota_{\vlon_n} (Y^n_{l_n})} \left[ \bigotimes^n_{j=1} \ \underbracket{\left( \mu_{\iota_{\vlon_j} (Y^j_1) }  \otimes \cdots \otimes \mu_{\iota_{\vlon_j} (Y^j_{l_j}) } \right) \left(J^1_{\iota_{\vlon_j} (Y^j_1) \ \ldots \ \iota_{\vlon_j} (Y^j_{l_j})} \right)^* H^1 (\alpha (Z_{T_j}))}\right]\\
	\left( J^1_{\alpha (\sigma_1) \  \ldots \   \alpha (\sigma_n)} \right)^* \ .
	\end{split}
	\end{equation}
	Now, $ Z_{T_j} = \left(J^{\gamma_{\vlon_j}}_{Y^j_1 \ \ldots \ Y^j_{l_j}}\right)^* \ \gamma_{\vlon_j} (f_j) \ J^{\gamma_{\vlon_j}}_{X^j_1 \ \ldots \ X^j_{k_j}} $.
	Thus, the $ j $-th tensor component (underlined) in the middle of the expression \ref{munat1} becomes
	\begin{equation}\label{munat2}
	\begin{split}
	\left(\kappa^2_{\vlon_j, Y^j_1} \otimes \cdots \otimes \kappa^2_{\vlon_j, Y^j_{l_j}}\right) \left(\kappa^1_{\vlon_j, Y^j_1} \otimes \cdots \otimes \kappa^1_{\vlon_j, Y^j_{l_j}}\right)^* \ \left(J^1_{\iota_{\vlon_j} (Y^j_1) \ \ldots \ \iota_{\vlon_j} (Y^j_{l_j})} \right)^* \ \left[H^1 \alpha \left(J^{\gamma_{\vlon_j}}_{Y^j_1 \ \ldots \ Y^j_{l_j}}\right)\right]^* \\
	H^1 \left(\iota_{\vlon_j} (f_j)\right) \ H^1 \alpha \left( J^{\gamma_{\vlon_j}}_{X^j_1 \ \ldots \ X^j_{k_j}} \right)
	\end{split}
	\end{equation}
	From the monoidal property of $ \kappa^i_\vlon : F_\vlon \ra H^i \circ \iota_\vlon$, we get another formula
	\begin{equation}\label{kmonoid}
	H^i \alpha \left(J^{\gamma_\vlon}_{Z_1 \ \ldots \ Z_n} \right) \ J^i_{\iota_\vlon (Z_1) \ \ldots \ \iota_\vlon (Z_n)} \ \left(\kappa^i_{\vlon , Z_1} \otimes \cdots \otimes \kappa^i_{\vlon , Z_n} \right) = \kappa^i_{\vlon , Z_1 \otimes \cdots \otimes Z_n } \ J^\vlon_{Z_1 \ \ldots \ Z_n} \ .
	\end{equation} 
	Applying Formula \ref{kmonoid} twice (namely, for $ i=1,2 $) on the expression \ref{munat2}, we get
	\begin{align*}
		&\ 
		\left(J^2_{\iota_{\vlon_j} (Y^j_1) \ \ldots \ \iota_{\vlon_j} (Y^j_{l_j})} \right)^* \ \left[H^2 \alpha \left(J^{\gamma_{\vlon_j}}_{Y^j_1 \ \ldots \ Y^j_{l_j}}\right)\right]^* \left(\kappa^2_{\vlon_j ,t(\tau_j)}\right) \left(\kappa^1_{\vlon_j ,t(\tau_j)}\right)^* \ H^1 \left(\iota_{\vlon_j} (f_j)\right) \\
		&  \hspace{13cm} H^1 \alpha \left( J^{\gamma_{\vlon_j}}_{X^j_1 \ \ldots \ X^j_{k_j}} \right)\\
		=& \ 
		\left(J^2_{\iota_{\vlon_j} (Y^j_1) \ \ldots \ \iota_{\vlon_j} (Y^j_{l_j})} \right)^* \ \left[H^2 \alpha \left(J^{\gamma_{\vlon_j}}_{Y^j_1 \ \ldots \ Y^j_{l_j}}\right)\right]^* \ H^2 \left(\iota_{\vlon_j} (f_j)\right) \ \left(\kappa^2_{\vlon_j ,t(\sigma_j)}\right) \left(\kappa^1_{\vlon_j ,t(\sigma_j)}\right)^* \\
		&  \hspace{13cm}  H^1 \alpha \left( J^{\gamma_{\vlon_j}}_{X^j_1 \ \ldots \ X^j_{k_j}} \right)\\
	\end{align*}
	where the last equality follows from naturality of $ \kappa^i_{\vlon_j} $.
	Again applying Formula \ref{kmonoid} twice and the equation $ \mu_{\iota_{\vlon} (Z) }= \kappa^2_{\vlon,Z} \left( \kappa^1_{\vlon,Z} \right)^*$, the last expression becomes
	\begin{equation*}
		\begin{split}
			\left(J^2_{\iota_{\vlon_j} (Y^j_1) \ \ldots \ \iota_{\vlon_j} (Y^j_{l_j})} \right)^* \ \left[H^2 \alpha \left(J^{\gamma_{\vlon_j}}_{Y^j_1 \ \ldots \ Y^j_{l_j}}\right)\right]^* \ H^2 \left(\iota_{\vlon_j} (f_j)\right) \ H^2 \alpha \left( J^{\gamma_{\vlon_j}}_{X^j_1 \ \ldots \ X^j_{k_j}} \right) \ J^2_{\iota_{\vlon_j}(X^j_1) \ \ldots \ \iota_{\vlon_j}(X^j_{k_j}) } \\
			\left(\mu_{\iota_{\vlon_j} (X^j_1)} \otimes \cdots \otimes \mu_{\iota_{\vlon_j} (X^j_{k_j})} \right) \left( J^1_{\iota_{\vlon_j}(X^j_1) \ \ldots \ \iota_{\vlon_j}(X^j_{k_j}) } \right)^* \ 
		\end{split}
	\end{equation*}
	which turns out to be
	\begin{equation}\label{munat3}
	\left(J^2_{\iota_{\vlon_j} (Y^j_1) \ \ldots \ \iota_{\vlon_j} (Y^j_{l_j})} \right)^* \ H^2 \alpha \left( Z_{T_j} \right) \ \mu_{\alpha (\sigma_j)} \ .
	\end{equation}
	Replacing the underlined part in expression \ref{munat1} by \ref{munat3}, we get
	\begin{align*}
		& \ J^2_{\iota_{\vlon_1} (Y^1_1) \ \ldots \ \iota_{\vlon_n} (Y^n_{l_n})} \left[ \bigotimes^n_{j=1} \left(J^2_{\iota_{\vlon_j} (Y^j_1) \ \ldots \ \iota_{\vlon_j} (Y^j_{l_j})} \right)^* \ H^2 \alpha \left( Z_{T_j} \right) \ \mu_{\alpha (\sigma_j)} \right]\ 
		\left( J^1_{\alpha (\sigma_1) \  \ldots \   \alpha (\sigma_n)} \right)^* \\
		= & \ J^2_{\alpha(\tau_1) \ \ldots \ \alpha(\tau_n)} \left[ \bigotimes^n_{j=1} \ H^2 \alpha \left( Z_{T_j} \right) \ \mu_{\alpha (\sigma_j)} \right]\
		\left( J^1_{\alpha (\sigma_1) \  \ldots \   \alpha (\sigma_n)} \right)^* \ \t { (using Formula \ref{Jformula})} \\
		= & \ J^2_{\alpha(\tau_1) \ \ldots \ \alpha(\tau_n)} \left[ \bigotimes^n_{j=1} \ H^2 \alpha \left( Z_{T_j} \right) \right]\ \left[ \bigotimes^n_{j=1} \ \mu_{\alpha (\sigma_j)} \right]\
		\left( J^1_{\alpha (\sigma_1) \  \ldots \   \alpha (\sigma_n)} \right)^*\\
		= & \ J^2_{\alpha(\tau_1) \ \ldots \ \alpha(\tau_n)} \left[ \bigotimes^n_{j=1} \ H^2 \alpha \left( Z_{T_j} \right) \right]\
		\left( J^2_{\alpha (\sigma_1) \  \ldots \   \alpha (\sigma_n)} \right)^* \ \mu_{\alpha(\sigma)} \ \t{ (since $ \mu_\alpha $ is monoidal)}\\
		= & \ H^2 \alpha \left( Z_{T_1} \otimes \cdots \otimes Z_{T_n}  \right) \ \mu_{\alpha(\sigma)} \ = \ \ H^2 \alpha \left( Z_T \right) \ \mu_{\alpha(\sigma)} \ . 
	\end{align*}
	\vskip 1em
	Finally, we have obtained a unitary natural monoidal isomorphism $ \mu_\alpha : H^1 \alpha \ra H^2 \alpha $ such that $ \mu_{\alpha \gamma_\vlon} \ \kappa^1_\vlon = \kappa^2_\vlon $.
	For defining $ \mu $ in the general form, consider $ (\sigma , p) \in \t{Obj} (\mcal{NCP})$ for $ \sigma \in \t{Obj}(\mcal{NCP}) $ and projection $ p \in \mcal{NCP} (\sigma,\sigma) $.
	Here also, there is only one choice (by naturality), namely
	\[
	\mu_{(\sigma , p)} \ \coloneqq \ H^2 (p) \ \mu_\sigma \ H^1 (p) \ \in \ \mcal C \left(H^1 (\sigma , p)) , H^2 (\sigma , p) \right) \ .
	\]
	Since $ \alpha $ is trivially monoidal, it almost comes for free that $ \mu $ is a unitary natural monoidal isomorphism compatible with $ \kappa^i_\vlon $ for $ i=1,2 $, $ \vlon = \pm $.
	Note that in our construction of $ \mu $, there is a unique choice at each stage.
	Hence, $ \mu $ has to be unique.
\end{proof}

\section{The free oriented extension of subfactor planar algebras}\label{FOE}

We will call a $ \Lambda $-oriented planar algebra simply an \textit{oriented planar algebra} if $ \Lambda $ is singleton.
This not only simplifies the terminology, but also agrees with the definition presented in \cite[Definition 1.2.7]{J1}.
In this section we will study oriented factor planar algebras and their relation to subfactor planar algebras.
Throughout this section, whenever we talk about oriented planar algebra, we assume $ \Lambda := \{ + \} $ and $ \ol \Lambda := \{ - \} $.
So, $ W = W_\Lambda $ will be the set of words with letters from $ \{\pm\} $ (note this $\pm$ has nothing to do with $\pm$ discussed in the preliminaries concerning free products).
Since $ \Lambda $ is a singleton, we do not label any of the (oriented) strings of any $ \Lambda $-oriented tangle; rather, we assume that each string is labeled with $ + $.
With this convention, a marked point on the external disc is assigned $ + $ or $ - $ according as the string attached to it has orientation towards or away from the point; for marked points on the internal discs, the convention is just the opposite.

In this context, it is important to talk about Jones' \textit{subfactor planar algebras} (\cite{J}).
We briefly recall the definitions.
Let $W_{\t {alt}} $ denote the set  of words having even length with $ + $ and $ - $ appearing alternately.
Then, one can consider oriented tangles where the colors of internal and external discs must belong to $ W_{\t {alt}} $ such that it is possible to put a checker-board shading; such tangles are called \textit{shaded tangles}.
If the color of the external or an internal disc in a shaded tangle is $ \varnothing $, then the region attached to boundary of the disc could be unshaded or shaded; we specify this by renaming the color of the disc as $ +\varnothing$ or $ -\varnothing $ respectively.
Let $ W_{ \t {alt}} $ contain the elements $ \pm \varnothing $ as well.
A subfactor planar algebra consists of a family of vector spaces $ \{P_w\}_{w\in W_{ \t {alt}}} $ on which the shaded tangles act satisfying properties analogous to that in Definitions \ref{opa} and \ref{pa types}.
In the original definition, Jones indexed the vector spaces by $ \{\vlon k : \vlon \in \{ \pm \}, k\in \N\} $ instead of $ W_{\t {alt}} $ (where $ \vlon k $ corresponds to the word of length $ 2k $ with alternate letters $ \pm $, beginning with $ \vlon $).

Given a oriented factor planar algebra $ Q $, since shaded tangles can be thought of as oriented tangles by simply forgetting the shading, we can canonically construct a subfactor planar algebra called its \textit{shaded part of $ Q $}, denoted $ \mcal{S}(Q) $.
For $w\in W_{alt}$, $\mcal{S}(Q)_{w}:=Q_{w}$.
The action of shaded tangles is simply the action of the associated oriented tangle.

Let  $\mcal{P}_{or}$ denote the category whose objects are oriented factor planar algebras and whose morphisms are $*$-planar algebra morphisms. Similarly, let $\mcal{P}_{sh}$ denote the category whose objects are subfactor planar algebras, and whose morphisms are $*$-planar algebra morphisms. Obviously the assignment $Q\mapsto \mcal{S}(Q)$ induces a functor $\mcal{S:}\mcal{P}_{or}\rightarrow \mcal{P}_{sh}$. 

\begin{defn} The functor $\mcal{S}:\mcal{P}_{or}\rightarrow \mathcal{P}_{sh}$, $Q\mapsto \mcal{S}(Q)$ is called the \textit{shading functor}.
\end{defn}

To understand this functor on the level of von Neumann algebras, an oriented factor planar algebra $Q$ corresponds to the rigid C*-tensor category generated by a single bimodule $\mcal H$ of a $\rm{II}_{1}$ factor $N$.
Taking alternating tensor powers of $\mcal H$ and $\overline{\mcal H}$ gives the standard invariant for the subfactor $N\subseteq M$, where $M$ is the $\rm{II}_{1}$ factor associated to the $Q$-system $\mcal H\otimes_{N} \overline{\mcal H}\in \t {Bim}(N)$.
The standard invariant $\mcal{S}(Q)$ is precisely the standard invariant of this subfactor.
Note that we cannot recover tensor powers of $\mcal H$ from this information.
In other words, the sufactor standard invariant forgets information. We are led naturally to the following definition.

\begin{defn}\label{oe defn}
Let $ P $ be a subfactor planar algebra. An \textit{oriented extension of $ P $} is a oriented factor planar algebra $ Q $ such that $ \mcal{S}(Q)$ is $ * $-isomorphic to $ P $.
	
\end{defn}

The set of (isomorphism classes of) oriented extensions of a subfactor planar algebra is precisely its pre-image under the functor $\mcal S$.

\begin{rem}
Note that subfactor planar algebras correspond to rigid, semisimple C*-2-categories $ \mcal B $ with two 0-cells $ \{+,-\} $ such that (i) tensor units in $ \mcal B_{++} $ and $ \mcal B_{--} $ are simple and (ii) there is a 1-cell $ \rho \in \mcal B_{+-} $ which tensor-generates the whole of 2-category $ \mcal B $.
If we call such 2-categories as \textit{singly generated}, with this correspondence one can also define oriented extension of the singly generated 2-category $ \mcal B $ as a singly generated C*-tensor category $ \mcal C $, generated by $ \sigma $, such that the underlying 2-category generated by $ \sigma $ is equivalent to $ \mcal B $.
\end{rem}

\subsection{The free oriented extension}
The first obvious question is whether an oriented extension always exists. We answer this affirmatively, by constructing a canonical one, called the \textit{free oriented extension}. 

To proceed with this construction, let $P$ be a subfactor planar algebra.
For every word $ w \in W $ let $ D_w $ be the set of oriented tangles in which the color of the external disc is $ w $ and all the internal discs (if any) have their colors in $ W_{ \t {alt}} $ along with a labelling of each internal disc with an element in the corresponding $ P $-space.
Note that the tangles are arbitrary oriented tangles such that the boundary conditions along any disc are alternating, but these tangles \textit{do not} need to admit a checker-board shading.
A typical element of $ D_w $ will be denoted by $ T(x_1, \ldots , x_n) $ where $ T $ is the unlabelled tangle with internal discs (if any) $ D_1, \ldots , D_n $ having colors $ w_1, \ldots,w_n $, labelled with $ x_1 \in P_{w_1} , \ldots , x_n \in P_{w_n} $ respectively.
An example of such an element with colors of internal discs $ w_1=+-+-, w_2= +-+-, w_3= -+-+, w_4=-+ $ and $ x_i \in P_{w_i}, i=1,2,3,4 $ is given in the following figure.
Note that, in this example, the external disc has $ 6 $ marked points on it and its color is $w_0= -+--++ $.
\begin{figure}[h]
	\psfrag{a}{$ \bigstar$}
	\psfrag{b}{$ \bigstar$}
	\psfrag{c}{$ \bigstar $}
	\psfrag{d}{$ \bigstar $}
	\psfrag{e}{$ x_1 $}
	\psfrag{f}{$ x_3 $}
	\psfrag{g}{$ x_2 $}
	\psfrag{h}{$ x_4 $}
	\psfrag{i}{$ \bigstar $}
	\psfrag{j}{$ D_0 $}
	\includegraphics[scale=0.4]{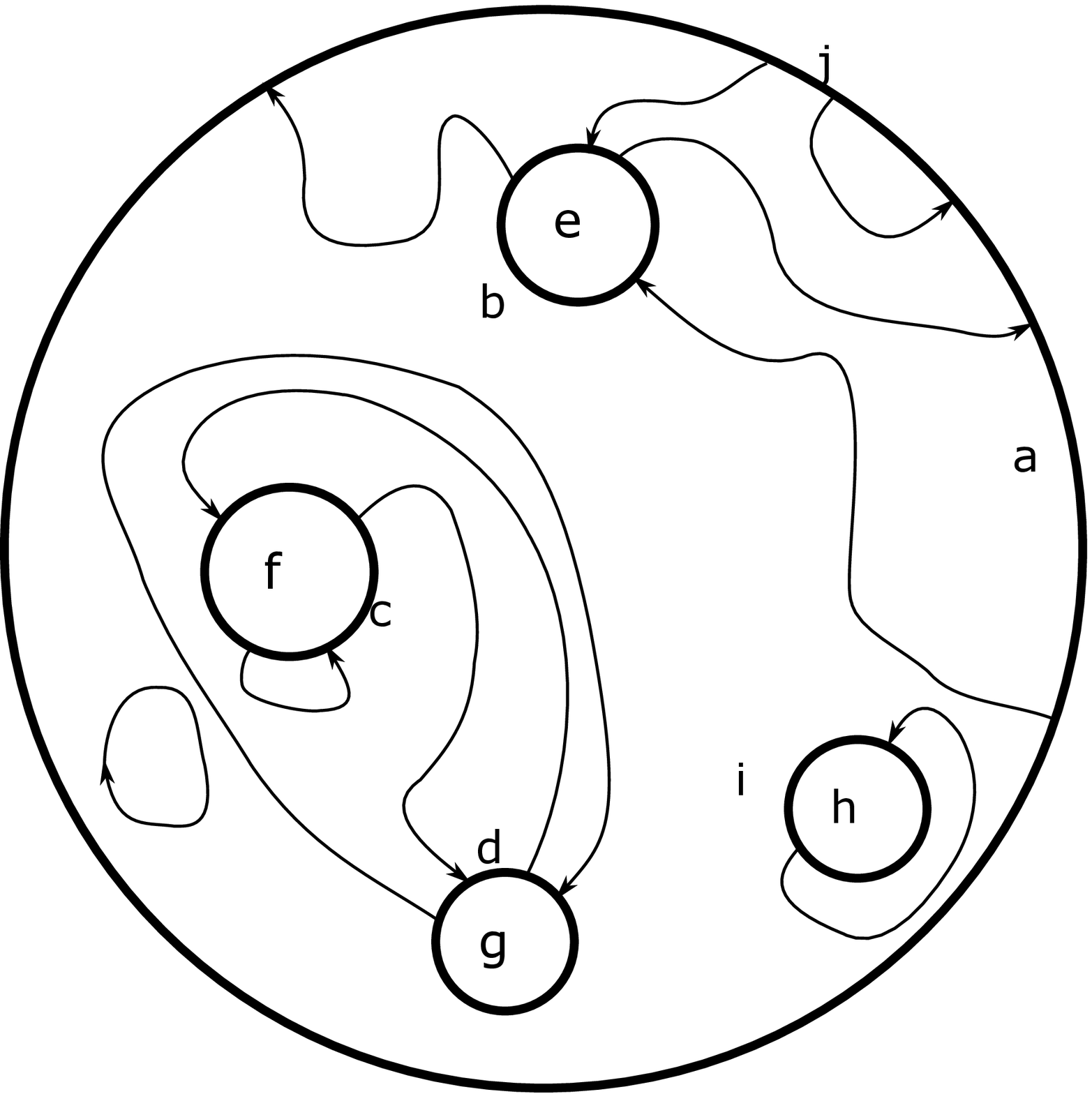}
\end{figure}

Set $ \mcal D_w := \C\text{-span}(D_w) $.
Note that we have an involutive map $ D_w \ni T(x_1, \ldots , x_n) \os * \longmapsto T^* (x^*_1, \ldots , x^*_n) \in D_{w^*} $; we extend this conjugate linearly to get an involution $ *: \mcal D_w \ra \mcal D_{w^*} $.
Observe that $\mcal D := \{\mcal D_w\}_{w\in W} $ is an oriented $ * $-planar algebra where the action of tangles on labelled tangles comes simply from composition.
However, $ \mcal D $ is far from being a oriented factor planar algebra, since at this point, the spaces $\mcal{D}_{w}$ are all infinite dimensional

In order to define a sesquilinear form on each of $ \mcal D_w $, we first define an evaluation map associated to an oriented tangle with external disc having color $ \varnothing $ and all internal discs having colors in $ W_{\t {alt}} $ labelled by elements of $P$; such labeled tangles will be referred as `networks' (see \cite{J}).
Topologically, a network $ N $ is a disjoint union of its connected components, and each component necessarily consists of closed, shaded, P-labelled tangles. These can be nested in the disc in complicated ways. However, for each connected component, one can forget the rest of the network and think of the tangle as being a closed shaded tangle with elements from $P$.
Now, $ P_{\pm \varnothing} $ are one dimensional algebras, and thus we can associate a scalar to each closed network.
For a network $N$, define $ Z(N) $ to be the product of the scalars arising from each connected component shaded tangle.
We call this the \textit{partition function}. Note that by construction, since $P$ is a spherical subfactor planar algebra this partition function is also spherical.

We define a sesquilinear form on $\mcal{D}_{w}$ by 
\[[X,Y]_{w}:=Z\left({H_w \circ (S , T^* )} (x_1, \ldots,x_m , y^*_1,\ldots,y^*_n)\right)\]

\noindent for $X= S(x_1, \ldots,x_m) ,\ Y=T(y_1,\ldots,y_n) \in D_{w}\subseteq \mcal{D}_{w}$, where $ H_w $ is the inner product tangle defined in Section \ref{prelim}. The following lemma is a crucial step in our construction.

\begin{lem}\label{positivity}
For all $w\in W$,  $[X,X]_w \geq 0$ for all $ X \in \mcal D_w $.
\end{lem}

\begin{proof}
Without loss of generality, we may assume that in the decomposition $ X = \us {i=1} {\os k \sum} c_i X_i $ with respect to the canonical basis $ D_w $, none of the $ X_i $'s  contain any non-empty network, that is, the union of the strings and the boundary of the discs (internal and external) is connected in each $ X_i $.
This automatically settles the case of $ w =\varnothing$.
Moreover, if $\varnothing \neq w \in W_{\t {alt}}$,
then all $ X_i $ are shaded and thereby the positivity of the $ P $-action implies $ [\cdot, \cdot]_w $ is positive semi-definite.

From now on, we will assume $ w\in W \setminus W_{\t {alt}} $.
Recall that $ D_w $ is the set of oriented tangles in which the color of external disc is $ w $ and all the internal discs (if any) have their colors in $ W_{\t{alt}} $.
Since all the internal discs (a)  have even number of marked points, and (b) have colors with equal number of $ + $ signs and $ - $ signs (as they belong to $ W_{\t{alt}} $), in order to have $ \mcal D_w\neq \{0\} $ (that is, $ D_w \neq \emptyset$,) the word $ w $ must be of even length with the same number of $ + $ signs and $ - $ signs.
Again, since $ w$ has same number of $ + $ signs and $ - $ signs and $ w \in W \setminus W_{\t {alt}} $, $ w $ must have a sub-word (in the non-consecutive sense) of the form $ (+,+) $ or $ (-,-)$, so it will be enough to consider the case when $ w  $ starts and end with same sign.
This is because the sesquilinear form is invariant under the action of rotation tangle.
More precisely, if $ w = (w_1,w_2) $ and $ \rho_{w_1,w_2} : (w_1 w_2) \ra (w_2 w_1) $ denotes the rotation tangle as described in the preliminaries, then $ [X,Y]_w = [\rho_{w_1,w_2} (X) , \rho_{w_1,w_2} (Y)]_{(w_2,w_1)} $ for all $ X,Y \in \mcal D_w $. Let 
$$ W' := \{w \in W \setminus W_{\t {alt}} : w \t { starts and ends with the same sign and } D_w \neq \emptyset\}. $$

\noindent Every $ w \in W' $ can be expressed as a unique concatenation $ w_1 w_2 \ldots  w_k $ of consecutive sub-words, where each $ w_i $ is a word with $ \pm $ appearing alternately such that the last sign of the sub-word $w_{i}$ matches with the first one of $w_j$; we will refer these special sub-words as `MAS' (which stands for \textit{maximally alternately signed}).

Observe that (i) each MAS sub-word of even length has equal number of $ + $ signs and $ - $ signs and (ii) each MAS sub-word of odd length starting and ending in $ + $ (resp. $ - $) has a $ + $ (resp. $ - $) more in number than that of $ - $ (resp. $ + $).
Since $ w $ has the same number of $ + $ signs and $ - $ signs, it follows that

\begin{enumerate}

\item the number of MAS sub-words of $ w $ with odd length starting and ending with $ + $ must be the same as that with $ - $; in particular, the number of odd length MAS sub-words must be even.
\\

\item there must be at least one MAS sub-word of even length.

\end{enumerate}

We will now prove that the total number of MAS sub-words of any $ w \in W' $ must be even.
This is clearly true if all the MAS sub-words of $ w $ have even length since $ w $ starts and end with the same sign.
So, let us assume $ w \in W' $ has both odd and even length MAS sub-words.
It will be useful to take a disc and arrange the signs in $ w $ as marked points on the boundary moving clockwise.
Note that if the last sign of any odd length MAS sub-word (a) differs from or (b) matches with the first sign of the very next odd length MAS sub-word moving clockwise, then the number of even length MAS sub-words in between must be (a) odd or (b) even respectively.
By (1) above, the number of instances of the case (a) is even.
Thus, in the end, the total number of even length MAS sub-words is even and so is the number of MAS sub-words.

Now, fix a $ w\in W' $.
Let $ w=w_1, \ldots,w_{2k} $ be the MAS sub-word decomposition.
Set $ w_{\t {odd}} := w_1w_3\ldots w_{2k-1} $ and $ w_{\t {even}} := w_2 w_4 \ldots w_{2k} $. We have the following assertion.

\medskip

\noindent\textbf{Assertion.} Every $X = T(x_1,\ldots,x_n) \in D_w $ which does not contain any non-empty network, can be expressed uniquely as an overlay of labelled tangles $ X_{\t {odd}} \in D_{w_{\t {odd}}} $ and $ X_{\t {even}} \in D_{w_{\t {even}}} $.\\

\noindent\textit{Proof of the assertion.}
First we consider the case in which there is no internal disc inside $ X =T \in D_w $. Then $T$ is a Temperley-Lieb (TL) diagram with color of the external disc being $ w \in W' $.
Note that any such TL diagram induces a non-crossing pairing of opposite signs in $ w $ exhausting all the signs; this puts a further restriction that two opposite signs coming from two MAS sub-words which are adjacent around the disc, can never be paired.

What we need to show is that two opposite signs can be paired only if either both belong to two even-indexed MAS sub-words or two odd-indexed ones.
We use induction on the length of $ w$.
The minimum length of elements in $ W' $ is $ 4 $ and there are exactly two words, namely, $ + - - + $ and $ - + + -  $.
In both instances, there are exactly two MAS sub-words.
Thus each MAS sub-word should be paired within itself.

For the inductive step, suppose $ w $ has length $ 2n $.
Let $ w= w_1 \ldots w_{2k} $ be the MAS sub-word decomposition.
In $ w $ there must exist two consecutive signs (namely, $ + -  $ or $  - +  $) which are paired by the TL diagram $ T $; let us denote this sub-word by $ v $.
Clearly, this $ v $ must appear in a MAS sub-word, say $ w_j $.
Let $ w' $ (resp., $ w'_j $) denote the word obtained by removing $ v $ from $ w $ (resp. $ w_j $),  and $ T' $ be the corresponding $ w' $-TL diagram obtained from $ T $.
If $ v $ is strictly smaller that $ w_j $, then we have the MAS sub-word decomposition $ w' = w_1 \ldots w_{j-1} w'_j w_{j+1} \ldots w_{2k} $.
Since $ \abs{w'} < \abs w $, using the inductive hypothesis, we can express $ T' $ as an overlay; we simply attach the pairing of $ v $ at the appropriate place to get the overlay of $ T $.
If  $ v = w_j $, then the MAS sub-word decomposition becomes $ w' = w_1 \ldots  w_{j-2} w''_{j-1} w_{j+2}\ldots w_{2k}) $ where $ w''_{j-1} = w_{j-1} w_{j+1} $.
By inductive hypothesis on $ T' $, we see that no pairing can occur between an even-indexed MAS sub-word of $ w' $ and an odd-indexed one; so the same holds for $ T $ as well.
For the case when $ v$ is $w_1$ or $ w_{2k} $, we simply apply a rotation to make $v$ interior and use the same argument.

For the general case, we assume $X =  T (x_1, \ldots,x_n) $ has internal disc(s).
We say two internal discs $ D_i $ and $ D_j $ of $ T $ are \textit{related} if there is a sequence of internal discs starting with $ D_i $ and ending with $ D_j $ such that any two consecutive internal discs in the sequence are connected by a string.
Since there is no non-empty network in $ X $, this clearly becomes an equivalence relation.
Fixing an equivalence class, we could use isotopy to bring all the internal discs in the class along with the strings connecting them inside a new disc whose boundary is intersected by the strings connecting these internal discs with the external one. The interior of this new disc is a labelled shaded tangle.
Without loss of generality, we may assume that all strings emanating from every internal disc in $ T $ go to the external one which has color $ w $.
By composing $ T $ with TL diagrams in all its internal disc, we get a TL diagram.
Since the assertion holds for TL diagrams, we may conclude that if a string from an internal disc connects to a sign in an odd (resp., even) indexed MAS sub-word of $ w $, then all other strings from the same disc should go to only odd (resp., even) indexed MAS sub-word.
That is all one needs to obtain the overlay mentioned in the statement of the assertion.

The uniqueness of the overlay holds because there is no network inside $ X $.
In particular, one obtains $ X_{\t {even}} $ (resp., $ X_{\t {odd}} $) simply by erasing all the marked points on the external disc corresponding to the odd (resp., even) indexed MAS sub-words along with the strings and internal discs connected to these marked points.\qed
\vskip 1em
We return to the proof of the lemma.
For every $ w \in W $ and $ X \in D_w $, let $ \lambda_X $ denote the product of the scalars corresponding to the $ P $-action of every connected networks in $ X $, and $ X' \in D_w $ be the element obtained by removing all networks in $ X $.

In order to establish positivity of $ [\cdot,\cdot]_w $ for $ w \in W' $, consider the linear map defined by
\[
\mcal D_w \supset D_w \ni X \os {\Phi_w} \longmapsto \lambda_X \left(X'_{\t {odd}} \otimes X'_{\t {even}}\right) \in \mcal D_{w_{\t {odd}}} \otimes \mcal D_{w_{\t {even}}}.
\]
Observe that if $ [\cdot , \cdot]_{w,\otimes} $ denotes the sesquilinear form on $ \mcal D_{w_{\t {odd}}} \otimes \mcal D_{w_{\t {even}}} $ obtained from the product of $ [\cdot , \cdot]_{w_{\t {odd}}} $ and $ [\cdot , \cdot]_{w_{\t {even}}} $, then $ [X,Y]_w = [\Phi_w (X), \Phi_w (Y)]_{w,\otimes}$ for all $ X,Y \in \mcal D_w $.

The map $ \Phi_w $ tells us that $ \dim (\mcal D_w) \leq \left[\dim (\mcal D_{w_{\t {odd}}}) \cdot \dim (\mcal D_{w_{\t {even}}})\right] $.
Since the lengths of both $ w_{\t {odd}} $ and $w_{\t {even}} $ are strictly smaller than that of $ w $, a simple induction on the length of $ w $ will tell that $ \mcal D_w $ is finite dimensional for all $ w \in W' $ and that the form is a positive definite as the tensor product of positive sesquilinear forms is again positive on finite dimensional vector spaces.
\end{proof}

For $ w_1,w_2 \in W $, a \textit{$ P $-labelled annular tangle from $ w_1 $ to $ w_2 $} is an oriented tangle in which the color of the external disc is $ w_2 $, there is an unlabelled distinguished internal disc with color $ w_1 $ and all other internal discs have colors in $ W_{\t {alt}} $ and labels from the corresponding $ P $-spaces.
Any such annular tangle $ A:w_1 \ra w_2 $ induces a linear map from $ \mcal D_{w_1} $ to $ \mcal D_{w_2} $ via composition; moreover, one can define an annular tangle $ A^\# :w_2 \ra w_1 $ which is obtained by (i) reflecting $ A $ around the external disc so that the external (resp. distinguished internal) disc becomes the distinguished internal (resp. external) disc after reflection, (ii) reversing the orientation of every string after reflection, and (iii) replacing the label of each internal disc by its $ * $.
It is easy to see that $ \# $ is an involution and $ [A (X),Y]_{w_2} = [X,A^\# (Y)]_{w_1} $ for all $ X \in \mcal D_{w_1} $, $ Y \in \mcal D_{w_2} $ (here we use sphericality of the partition function $Z$.)

Following \cite{J}, we define $\mcal{J}_{w}\subseteq \mcal{D}_{w}$ by $\mcal{J}_{w}:=\{x\in  \mcal{D}_{w}\ :\ Z(A(x) )=0\ \text{for all}\ A:w\rightarrow \varnothing\}$.
By \cite[Proposition 1.24]{J}, this is a planar ideal of $\mcal{D}$, and clearly in our context this is a *-ideal.
We claim that $X\in \mcal{J}_{w}$ if and only if $X$ is in the kernel of $[\cdot, \cdot]_{w}$. Certainly if $\mcal{J}_{w}$ is in the kernel of our form. Suppose $[X,X]_{w}=0$. Let $A:w\rightarrow \varnothing$ be a P-labelled annular tangle. Then by Cauchy-Schwartz, we have

$$|Z(A(X))|:=|[A(X), \varnothing]_{\varnothing}|=|[X, A^{\#}(\varnothing)]_{w}|\le [X,X]^{\frac{1}{2}}_{w} [A^{\#}(\varnothing), A^{\#}(\varnothing)]^{\frac{1}{2}}_{w} =0.$$ 

\noindent proving the claim. Therefore we can define the planar algebra $\mcal{F}(P)_w :=\mcal{D}_{w}/\mcal{J}_{w}$.
This is non-zero, since for $w\in W_{alt}$, $\mcal{D}_{w}/\mcal{J}_{w} \cong P_{w} \ne \{0\}$.
By \cite[Proposition 1.33]{J}, this is a C*-planar algebra.

\begin{defn}
The oriented factor planar algebra $\mcal{F}(P)$ is called \textit{free oriented extension} of the subfactor planar algebra $ P $.
We denote the (subfactor planar algebra) isomorphism between $ \iota: P \lra \mcal{S}(\mcal{F}(P)) $ by $ \left(\iota_w \colon P_w \lra \mcal{F}(P)_w \right)_{w \in W_{\t {alt}}} $.
\end{defn}

Note that the free oriented extension $\mcal{F}$ is actually a functor. Namely, if $\varphi: P\rightarrow P^{\prime}$ is a planar *-homomorphism, then the obvious definition $\mcal{F}(\varphi): \mcal{F}(P)\rightarrow \mcal{F}(P^{\prime})$ works. Simply define $\widetilde{\varphi}: \mcal D^P_{w}\rightarrow \mcal D^{P'}_{w}$, and check that it preserves the partition function. One of the motivations for studying this functor is that it is a left adjoint to the shading functor $\mcal{S}$. We express this via the following universal property.

\begin{thm}\label{free con any}
Let $P$ be a subfactor planar algebra and $Q $ an oriented factor planar algebra. For any $*$-homomorphism  $\psi: P\rightarrow \mcal{S}(Q) $, there exists a unique $ * $-homomorphism $ \widetilde \psi : \mcal{F}(P) \ra Q $ such that $ \widetilde  \psi \circ \iota   = \psi$.
\end{thm}

\begin{proof}

 Let $\psi=\left(\psi_w \colon P_w \ra Q_w \right)_{w \in W_{\t {alt}}} $.
We will use the notations set up in the construction of the free oriented extension.	
Define
\[
\mcal D_w \supset D_w \ni T(x_1,\ldots, x_n)\oset {\widehat \psi_w} {\longmapsto} Q_T(\psi_{w_1} x_1,\ldots, \psi_{w_n}  x_n) \in Q_w
\]
where $ T:(w_1,\ldots, w_n)\rightarrow w  \t{ and }  x_i \in P_{w_i} $, $ i=1,\ldots, n $ and extend $ \widetilde \psi_w $ linearly to $ \mcal D_w $.
From the very definition of the positive semi-definite form $ [\cdot,\cdot]_w $ on $ \mcal D_w $, the map $ \widehat \psi_w $ takes it to the inner product in $ Q_w $ induced by the $ Q $-action of the inner product tangle $ H_w $ (since the shaded part of $ Q $ is $ P $ via $ \psi $).
So, $ \widehat \psi_w $ factors through the quotient $ \mcal{F}(P)_w $ producing the map $ \widetilde \psi_w : \mcal{F}(P)_w \ra Q_w $.
Moreover, $ \widetilde \psi $ preserving the action is almost immediate.

To check the equation $ \widetilde  \psi_w  \circ \iota_w = \psi_{w}$ for all $ w \in W_{\t {alt}} $, take the element $ T(x_1, \ldots ,x_n) $ in the previous paragraph with the extra assumption $ w \in W_{\t {alt}} $.
We may assume $ T $ has no non-empty network (which anyway gives scalar), and thereby $ T $ becomes a shaded tangle.
Consider the equivalence class $ [T(x_1, \ldots, x_n)] \in \mcal{F}(P)_w $; note that $ \iota  P_T (x_1, \dots, x_n) = [T(x_1, \ldots , x_n)]  $.
On the other hand, \[\widetilde \psi_w [T(x_1, \ldots , x_n)] =  Q_T (\psi_{w_1} x_1, \ldots, \psi_{w_n} x_n) = \psi_w P_T (x_1,\ldots, x_n).\]
Since elements of the form $ P_T(x_1,\ldots,x_n) $ span $ P_w $, we have the desired equation.

For uniqueness, consider a $ * $-homomorphism $ \vphi : \mcal{F}(P) \ra Q $ satisfying $ \vphi_v \circ \iota_v = \psi_{v} $ for all $ v \in W_{\t {alt}} $.
Again, consider $ T(x_1, \ldots ,x_n) $ as above with $ w $ possibly not in $ W_{\t {alt}} $.
Observe that 
\[
\vphi_w [T(x_1, \ldots , x_n)] = \vphi_w \mcal{F}(P)_T (\iota_{w_1} x_1, \ldots , \iota_{w_n} x_n)] =  Q_T (\psi_{w_1} x_1, \ldots, \psi_{w_n} x_n) = \widetilde \psi_w [T(x_1, \ldots , x_n)]
\]
where the first equality follows from the definition of $\mcal{F}(P)$, second follows from $ \vphi : \mcal{F}(P) \lra Q $ being a $ * $-homomorphism and the equation satisfied by $ \vphi $ and $ \psi $, and the final equality comes from the definition of $ \widetilde \psi $.
\end{proof}

The above universal property makes the functor $\mcal{F}: \mcal{P}_{sh}\rightarrow \mcal{P}_{or}$ into a left adjoint of this functor $\mcal{S}$ (see \cite[Chap \rm{IV}, Theorem 2]{M}).

\subsection{Concrete realization}

Above, we constructed the free oriented extension from scratch whereas Thoerem \ref{free con any} gives us a universal property for an oriented extension being isomorphic to the free one.
In the next theorem, we will provide another useful construction of the free oriented extension. Given some oriented extension, we will show that one can find the \textit{free oriented extension} inside a certain free product category; we use techniques similar to those appeared in \cite{MPS}.
We will make use of this in the next section when we study hyperfinite realizability.

First, we describe a general construction for producing new oriented extensions from a given one.
Let $ \mcal C $ be a strict rigid semisimple C*-tensor category with simple tensor unit and tensor-generated by $ X_+ $; suppose $Q \coloneqq P^{X_+}$  and $ P\coloneqq \mcal S (Q) $.
Let $ g_+ $ be an element in a  group $G$ and $\mcal{B}:=\operatorname{Hilb}_{f.d.}(G)$ be the rigid C*-tensor category of $G$-graded finite dimensional Hilbert spaces.
We look at the following objects in the free product $ \mcal B \ast \mcal C $
\[
X_- \coloneqq \ol X_+ , \ \ g_- \coloneqq g^{-1}_+ , \ \ Y_+ \coloneqq g_+\ X_+\ g_+ , \ \ Y_- \coloneqq g_-\  X_-\  g_- = \ol Y_+ .
\]
Appealing to \Cref{freeprod}, although we will continue using the notation $ \mcal B \ast \mcal C $, all our objects and morphisms in the rest of this subsection will come from the corresponding full subcategory $ \mcal{NCP} $.
For $ w = (\vlon_1,\ldots , \vlon_n) \in W $, suppose $ X_w $ (resp., $ Y_w $) denotes the object $ X_{\vlon_1} \otimes \cdots \otimes X_{\vlon_n} $ (resp., $ Y_{\vlon_1} \otimes \cdots \otimes Y_{\vlon_n} $).
Choose unitaries $ R_-: \mathds 1 \ra g_- \otimes g_+$, $R_+ : \mathds 1 \ra g_+ \otimes g_- $ solving the conjugate equations for $ (g_+, g_-) $.
Note that, since $ g_{\pm} $ are invertible (and hence simple), these solutions are automatically standard, and they being unitaries, are normalized.
If an alternately signed word $ u $ (possibly of odd length) starts with $ \vlon $ and ends with $ \nu $, then we say that \textit{$ u $ is of $ (\vlon, \nu) $-type}; moreover, for such $ u $, we have the unitary $ \alpha_u \in \left(\mcal B \ast \mcal C\right) (Y_u \ , \ g_\vlon \ X_u \ g_\nu ) $ applying the unitaries $ R^*_- $ (resp., $ R^*_+ $) on each $ g_- \otimes g_+  $ (resp., $ g_+ \otimes g_- $) appearing between $ X_- $ and $ X_+ $ (resp., $ X_+ $ and $ X_- $) which appear in the tensor expansion of $ Y_u $.
For $ w \in W_{\t {alt}} $ of $ (\vlon, - \vlon) $-type, define the map $ \gamma_w :P_w \ra P^{Y_+}_w $ by\\
\[
P_w = \mcal C(\mathds 1, X_w)  \ni f \os {\gamma_w} \longmapsto \alpha^*_w \ (1_{g_\vlon} \otimes f \otimes 1_{g_{-\vlon}})\ R_\vlon  \ \in P^{Y^+}_w.
\]

For instance, if $ u=+-+ $, then, graphically, \[ \alpha_u =
\psfrag{a}{$ g_+ $}
\psfrag{b}{$ X_+ $}
\psfrag{c}{$ g_+ $}
\psfrag{d}{$ g_- $}
\psfrag{e}{$ X_- $}
\psfrag{f}{$ g_- $}
\psfrag{g}{$ g_+ $}
\psfrag{h}{$ X_+ $}
\psfrag{i}{$ g_+ $}
\raisebox{-1.7em}{\includegraphics[scale=0.24]{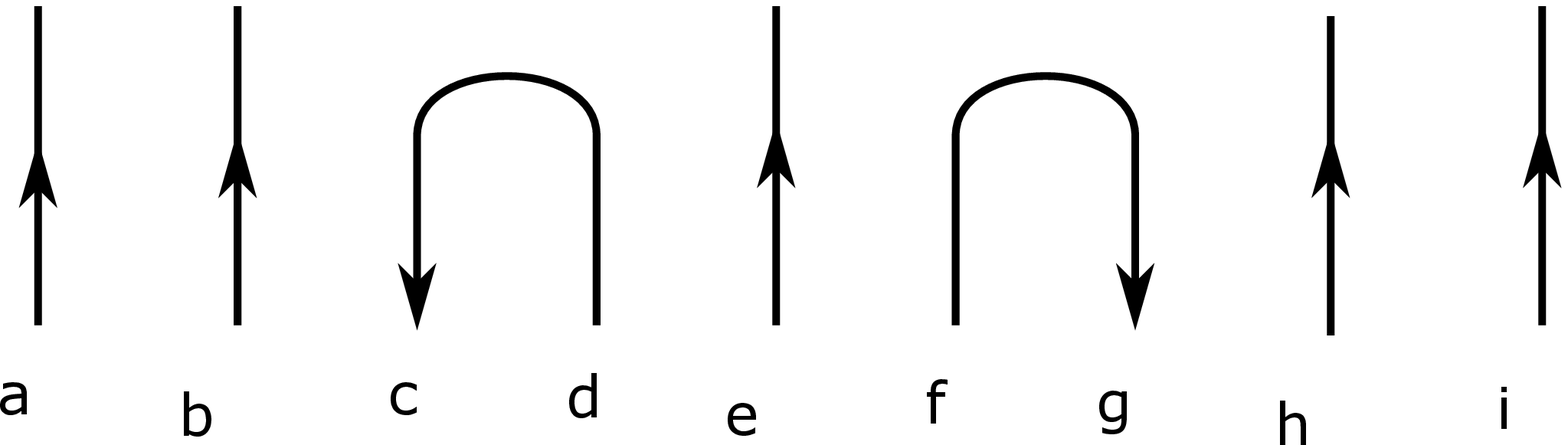}}\]

\noindent and, if $ w=+-+- $, then for $ f \in P_w = \mcal C(\mathds 1, X_w) $, $ \gamma_w(f) $ looks like

\[
\psfrag{a}{$ g_+ $}
\psfrag{b}{$ X_+ $}
\psfrag{c}{$ g_+ $}
\psfrag{d}{$ g_- $}
\psfrag{e}{$ X_- $}
\psfrag{f}{$ g_- $}
\psfrag{g}{$ g_+ $}
\psfrag{h}{$ X_+ $}
\psfrag{i}{$ g_+ $}
\psfrag{j}{$ g_- $}
\psfrag{k}{$ X_- $}
\psfrag{l}{$ g_- $}
\psfrag{m}{$ f $}
\raisebox{-1.7em}{\includegraphics[scale=0.24]{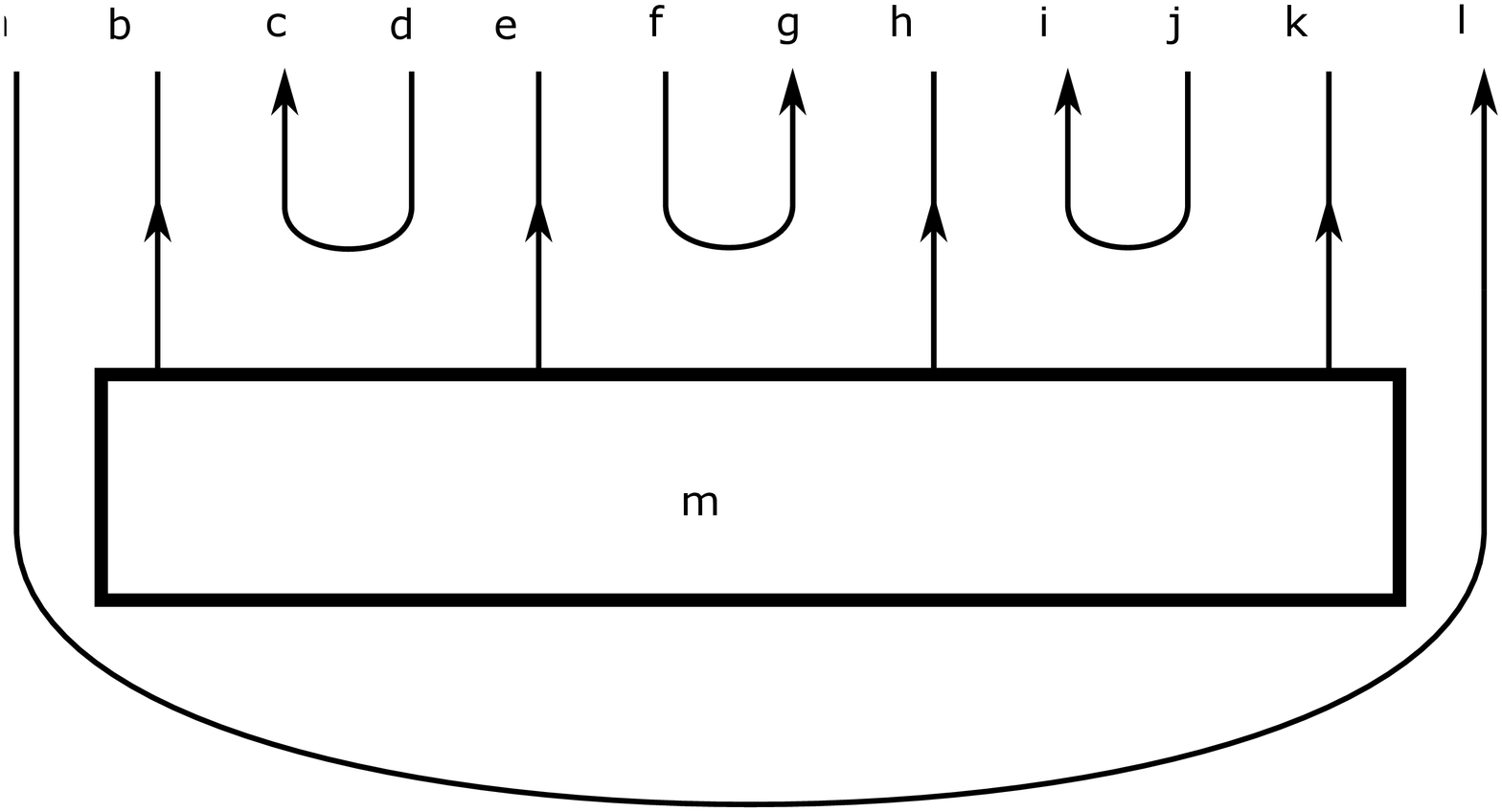}}\]

In order to build the planar algebra $ Q = P^{X_+} $ (and hence $ P $ too), we need to fix a normalized standard solution $ (R_{X_+} , \ol R_{X_+}) $ to the conjugate equations for $ (X_+ , X_-) $ (see Section \ref{C*2OPA}).
Choosing such solutions for $ (Y_+, Y_-) $ involving $ (R_{X_+} , \ol R_{X_+}) $ and $ (R_+, R_-) $ in the most obvious way, one can easily derive that $ \left(\gamma_w : P_w \lra P^{Y_+}_w\right)_{w \in W_{\t {alt}}} $ is indeed a subfactor planar algebra isomorphism. As we will see in the next lemma, for \textit{any oriented extension} $Q$, this oriented extension is isomorphic to the free oriented extension in the case $g_+$ has infinite order.

\begin{thm}\label{orientedrealizationivertible}
Let $ Q $ be any oriented extension of a subfactor planar algebra $ P $.
Then the projection category of $ \mcal F (P) $ (that is, $ \mcal K(\mcal C^{\mcal{F}(P)}) $) is equivalent to a full tensor-subcategory of the free product $ \operatorname{Hilb}_{f.d.}(\mathbb{Z}) \ast \mcal K(\mcal C^Q) $.
\end{thm}

\begin{proof}
Without loss of generality (via the converse part of \Cref{OPA-C*-correspondence}), we may assume $ Q \coloneqq P^{X_+}  $ and $ P \coloneqq \mcal S  (Q) $ for some strict, semisimple, rigid C*-tensor category $ \mcal C $ with simple tensor unit and tensor-generated by $ X_+ $.
We borrow the notations described right before the statement of this theorem by setting $ G = \Z $ and $ g_+ = 1 \in \Z $.
It is enough to show that $ P^{Y_+} $  is isomorphic to $ \mcal{F}(P) $.

Applying \Cref{free con any}, we get a unique $ * $-homomorphism $ \widetilde \gamma :\mcal{F}(P) \lra P^{Y_+} $ such that $ \widetilde \gamma_w \circ \iota_w = \gamma_w $ for all $ w \in W_{\t {alt}} $.
It is enough to show $ \widetilde \gamma_w $ is surjective for all $ w\in W $ (which implies $ \mcal{F}(P) $ and $ P^{Y_+} $ are isomorphic since $ * $-homomorphisms between oriented factor planar algebras are automatically injective).
We are already done with the cases when $ w \in W_{\t {alt}} $ (since $ \gamma_w $ is surjective) or $ P^{Y_+}_w $ is zero.

For $ w , w' \in W$, we say \textit{$ w $ is  a rotation of $ w' $} if there exists $ w_1,w_2 \in W $ such that $ w= w_1 w_2 $ and $ w' = w_2 w_1 $.
For any such $ w , w' \in W $, if $ \widetilde \gamma_w $ is surjective, then so is $ \widetilde \gamma_{w'} $.
To see this, pick any rotation implementing oriented tangle $ \rho: w \lra w' $; note that
\[
P^{Y_+}_{w'} = P^{Y_+}_\rho (P^{Y_+}_{w})  = P^{Y_+}_\rho \ \widetilde \gamma_w (\mcal{F}(P)_{w}) = \widetilde \gamma_{w'} \ \mcal{F}(P)_\rho (\mcal{F}(P)_{w}) = \widetilde \gamma_{w'} (\mcal{F}(P)_{w'}).
\]
Clearly, the rotation class of every word in $W \setminus W_{\t {alt}} $ whose $ P^{Y_+} $-space is non-zero, must have at least one word belonging to the set
\[
W_0 \coloneqq \left\{ w \in W \setminus W_{\t {alt}}  : w \t { starts and ends with the same sign, and } P^{Y_+}_w \neq \{0\} \right\} .
\]
Our goal boils down to establishing surjectivity of $ \widetilde \gamma_w $ for all $ w \in  W_0 $.
For this, we use induction on the number of MAS sub-words.

Let $ h_+ \coloneqq g_+ \otimes g_+ $ and $ h_- \coloneqq g_- \otimes g_- = \ol h_+ $.
Suppose $ w \in W_0 $.
If $w = w_1 w_2 \ldots w_k $ is the MAS sub-word decomposition and $ w_i $ has $ (\vlon_i,\vlon_{i+1}) $-type for $ 1\leq i \leq n $ (and thereby $ \vlon_1 = \vlon_{n+1} $), then we have an isomorphism

\noindent $\left(\mcal B \ast \mcal C\right) (\mathds 1 \ ,\ h_{\vlon_1} X_{w_1} h_{\vlon_2} X_{w_2} \cdots h_{\vlon_n} X_{w_n} ) \ni a \os {\displaystyle \sigma} \longmapsto$
\begin{flushright}
$ (R^*_{-\vlon_1}\otimes \alpha^*_{w_1} \otimes \cdots \otimes \alpha^*_{w_n}) (1_{g_{-\vlon_1}} \otimes a \otimes 1_{g_{\vlon_1}} ) R_{-\vlon_1} \in \left(\mcal B \ast \mcal C\right) (\mathds 1 \ , \ Y_{w_1} \otimes \cdots \otimes Y_{w_n}) = P^{Y_+}_w$.
\end{flushright}
Graphically, $ \sigma(a) $ is given by

\[
\psfrag{a}{$ Y_{w_1} $}
\psfrag{b}{$ Y_{w_2} $}
\psfrag{c}{$ Y_{w_n} $}
\psfrag{d}{$ \alpha^*_{w_1} $}
\psfrag{e}{$ \alpha^*_{w_2} $}
\psfrag{f}{$ \alpha^*_{w_n} $}
\psfrag{g}{$ g_{\varepsilon_1} $}
\psfrag{h}{$ g_{\varepsilon_1}  $}
\psfrag{i}{$ X_{w_1} $}
\psfrag{j}{$ g_{\varepsilon_2}  $}
\psfrag{k}{$ g_{\varepsilon_2}  $}
\psfrag{l}{$ X_{w_2}  $}
\psfrag{m}{$ a $}
\psfrag{n}{$ g_{\varepsilon_3}  $}
\psfrag{o}{$ g_{\varepsilon_{n-1}} $}
\psfrag{p}{$ X_{w_n}  $}
\includegraphics[scale=0.3]{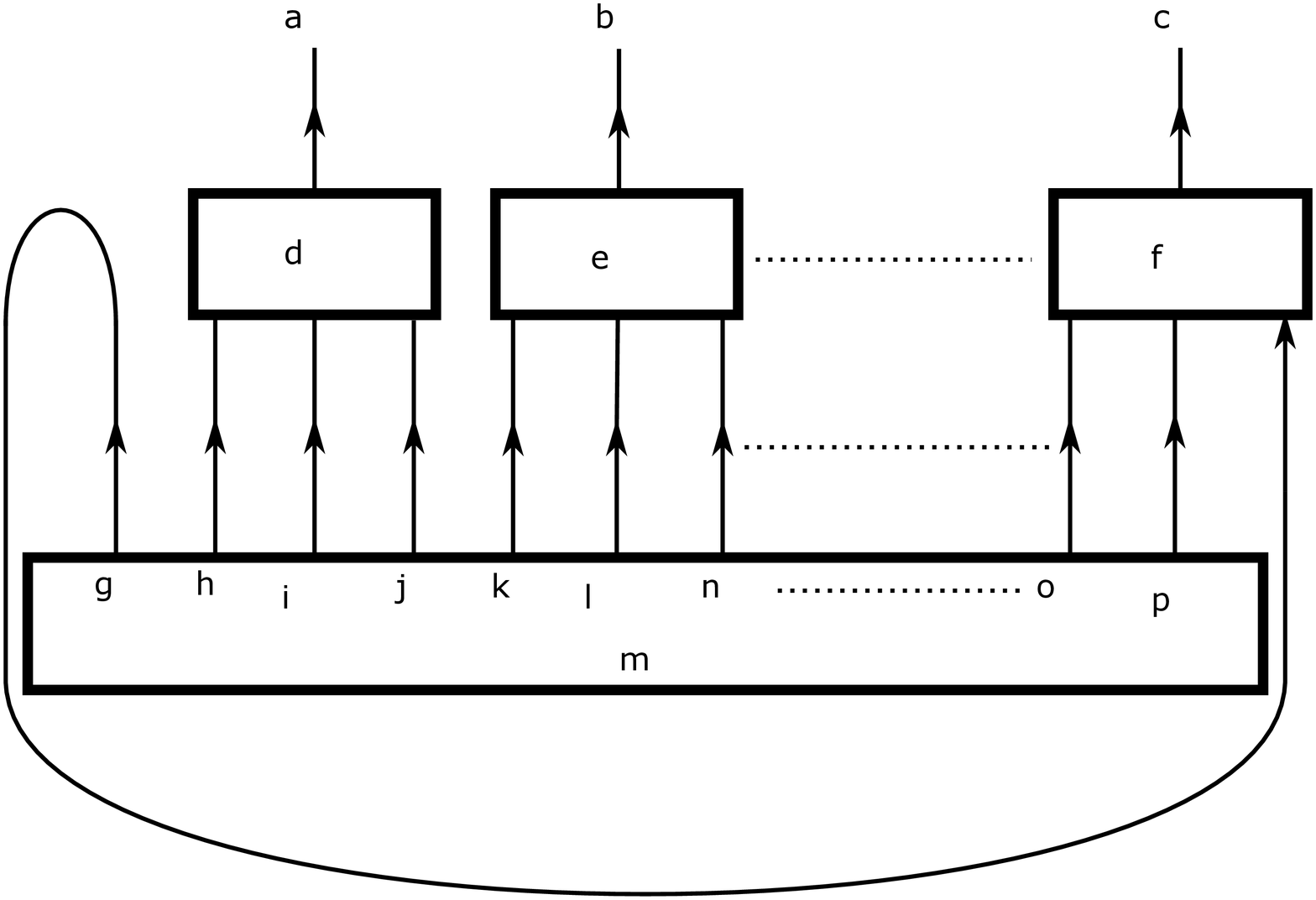}
\]

We now use the description of morphism spaces of $ \mcal B \ast \mcal C $ (in fact, $ \mcal {NCP} $) in \Cref{free prod prel}; the morphism space $ \left(\mcal B \ast \mcal C\right) (\mathds 1 \ ,\ h_{\vlon_1} X_{w_1} h_{\vlon_2} X_{w_2} \cdots h_{\vlon_n} X_{w_n} ) $ by definition is a subspace of
\[
\mcal B(\mathds 1_{\mcal B} , h_{\vlon_1} \otimes \cdots \otimes h_{\vlon_n})  \ \otimes \ \mcal C(\mathds 1_{\mcal C} , X_{w_1} \otimes \cdots \otimes X_{w_n}) = \mcal B(\mathds 1_{\mcal B} , h_{\vlon_1} \otimes \cdots \otimes h_{\vlon_n})  \ \otimes \ \mcal C(\mathds 1_{\mcal C} , X_w) .
\]

\comments{
Moreover, this morphism space has a distinguished spanning set whose typical element is described by\\
(1) a non-crossing partitioning of a rectangle (including the interior) whose top boundary has marked points labelled from left to right by the sequence $ h_{\vlon_1}, X_{w_1}, h_{\vlon_2}, X_{w_2}, \ldots, h_{\vlon_n}, X_{w_n} $ such that each partition has at least one member from the sequence and consists of only $ h_{\vlon_i} $'s or only $ X_{w_i} $'s, and\\
(2) each partition is assigned a non-zero morphism from $ \mathds 1 $ to the tensor product of the objects appearing in the partition read from left to right.

Pick any non-zero $ f $ in this spanning set.
}
Let $ f $ be a $ \left(\varnothing \  , \ h_{\vlon_1} X_{w_1} h_{\vlon_2} X_{w_2} \ldots h_{\vlon_n} X_{w_n} \right) $-NCP such that $ Z_f $ is non-zero.
Observe that $ f $ can viewed as an non-crossing overlay of a pair of oriented tangles $ S $ and $ T(x_1 , \ldots , x_m) $ such that 

\begin{enumerate}
\item
$S$ connects to the points on the boundary of the rectangle marked by $ h_{\vlon_i} $'s,  and has  internal discs labelled by non-zero elements of $ P^{h_+}_\bullet $,

\item
$ T $ connects to the points marked by $ X_{w_j} $'s with internal discs labeled by $ x_k $'s coming from $ P^{X_+}_\bullet $.

\end{enumerate}
Since $ (h_+, h_-) $ is a dual pair of invertible objects and none of their non-zero finite tensor powers is equivalent to $ \mathds 1 $, we may replace $ S $ (up to a non-zero scalar) by a Temperley-Lieb diagram where $ h_+ $ can be joined by a string only to $ h_- $.
In other words, the partitions in $ f $ consisting of $ h_{\vlon_i} $'s can be assumed to be pair-partitions of $ h_+$ and $ h_- $.

\comments{
As an element of a vector space,
\[
Z_f = P^{h_+}_S \  \otimes \ P^{X_+}_T (x_1, \ldots , x_m) \ \in \mcal B(\mathds 1_{\mcal B} , h_{\vlon_1} \otimes \cdots \otimes h_{\vlon_n})  \ \otimes \ \mcal C(\mathds 1_{\mcal C} , X_w).
\]
\vskip 2em

}

Now recall the definition of $D_{w}$ from the free oriented extension construction in the previous section. We claim that $ T(x_1, \ldots , x_m) \in D_w $ . To see this, we use induction on the number of MAS-sub-words.
First of all, note that $ n $ has to be even since $ S $ is given by pair-partitions, each consisting of $ h_+$ and $ h_-  $, implying
\[
\left\{i\in \{1, \ldots n\} : h_{\vlon_i} = +\right\} = \left\{ i\in \{1, \ldots n\} : h_{\vlon_i} = -\right\}.
\]
So, the smallest $ n $ is $ 2 $ in which case $ (h_{\vlon_1} , h_{\vlon_2}) $ is either $ (+,-) $ or $ (-,+) $.
This implies both $ w_1 $ and $ w_2 $ has to be even and thereby lie in $ W_\t{alt} $.
Now, the non-crossing nature of the partitioning forces $ X_{w_1} $ and $ X_{w_2} $ to be separate singleton partitions because $ (h_{\vlon_1} , h_{\vlon_2} )$ forms a partition.
Thus $ T $ has exactly $ 2 $ internal discs with colors $ w_1, w_2 \in W_\t{alt} $.
Hence, $ T(x_1 , x_2) \in D_w $.

Suppose our $ T(x_1, \ldots , x_m) \in D_w $ holds for all $T$, $x_{i}$, and $w\in W_0 $ with number of MAS sub-words at most $ 2n$.
Pick $ w \in W_0 $ with $ 2n+2 $ MAS sub-words.
In the Temperley-Lieb diagram $ S $, we can find consecutive elements $ h_{\vlon_i} $ and $ h_{\vlon_{i+1} }$ which are pair partitioned, implying $ \vlon_{i+1} = - \vlon_i $.
Further, we may assume $ i >1 $ since $ 2n+2 >4 $.
As a result, $ w_i $ must have even length and thereby belong to $ W_\t{alt} $.
The non-crossing partitioning forces $ X_{w_i} $ to become a singleton partition.
So, $ T  $ has an internal disc with the color $ w_i \in W_\t{alt} $, connected to the MAS sub-word $ w_i $ of $ w $ on the boundary of the external disc, and labelled with $ \widetilde x \in P^{X_+}_{w_i} = \mcal C (\mathds 1 ,X_{w_i}) $.
Set $ w' \coloneqq w_{i-1} w_{i+1}$ and $ w'' \coloneqq w_1,\ldots w_{i-2}, w', w_{i+2}, \ldots w_{2n+2} $.
Clearly, the word $ w' $ is alternately signed and the defining equation of $ w'' $ gives its MAS sub-word decomposition.
In the non-crossing partitioning of $ f $, we erase the partitions $ (X_{w_i}) $ and $ (h_{\vlon_i},h_{\vlon_{i+1}}) $ and their associated morphisms, to get a new one, say $ f' $.
Note that $Z_{f'} \in \left(\mcal B \ast \mcal C\right) (\mathds 1 , h_{\vlon_1} X_{w_1} \cdots X_{w_{i-2} } h_{\vlon_{i-1}} X_{w'} h_{\vlon_{i+2}} X_{w_{i+2}} \cdots X_{w_{2n+2}}) $.
We have the formula
\[
Z_f = \left( 1_{h_{\vlon_1} X_{w_1} \cdots X_{w_{i-2} } h_{\vlon_{i-1}} X_{w_{i-1}}} \otimes \widetilde x \otimes 1_{X_{w_{i+1}} h_{\vlon_{i+2}} X_{w_{i+2}} \cdots X_{w_{2n+2}}} \right) \circ Z_{f'}.
\]
Therefore $ Z_{f'} \neq 0 $ ($Z_f$ is assumed to be nonzero).
Let $ S' $ and $ T'(x'_1 , \ldots , x'_{m'}) $ be the corresponding tangles coming from $ f' $.
As $ w'' $ has $ 2n $ MAS sub-words, by induction hypothesis, we have $ T' (x'_1, \ldots , x'_{m'})  \in D_{w''}$.
Now, $ T(x_1, \ldots , x_m) $ can be obtained from $ T' (x'_1,\ldots , x'_{m'}) $ first by splitting $ w' $ on the boundary of the external disc into $ w_{i-1} $ and $ w_{i+1} $  and inserting the word $ w_i \in W_\t{alt} $ in between, and then attaching an internal disc of color $ w_i $ labelled with $ \widetilde x $ to this inserted word $ w_i $ with strings.
Hence, $ T(x_1 , \ldots , x_m) \in D_w$ proving the claim.

To complete the proof, it will suffice to show $ \sigma $ sends the non-zero $ Z_f $ to
\[
P^{Y_+}_T (\gamma_{v_1} x_1 , \dots , \gamma_{v_m} x_m ) = \widetilde \gamma [T(x_1, \ldots , x_m)] \in P^{Y_+}_w = \left(\mcal B \ast \mcal C\right) (\mathds 1 , Y_w)\subset \mcal B(\mathds 1_{\mcal B} , h_w) \otimes \mcal C ( \mathds 1_{\mcal C} , X_w).
\]
Note that $ \mcal B (\mathds 1_{\mcal B} , h_w) $ is one-dimensional.
Just as $ Z_f = P^{h_+}  \otimes P_T(x_1,\ldots , x_m)$, a its straightforward to see that $ \sigma Z_f $ can also be expressed as $ P^{g_+}_{S'} \otimes P^{X_+}_T (x_1,\ldots , x_m) $ where $ S' $ is a pair-partitioning of the $ g_\pm $ appearing in $ Y_w $.

Next, we look at the non-crossing partitioning $ P^{Y_+}_T (\gamma_{v_1} x_1 , \dots , \gamma_{v_m} x_m )$. 
The way $ T(x_1, \ldots , x_m) $ is read off from the non-crossing partition view of $ f $, we can say that the colors $ v_1 , \ldots , v_m$ (lying in $W_{\t{alt}} $) correspond to the partitions consisting of $ X_{w_i} $'s.
At this point, it is useful to work with a standard form representative of $ T $ (as described in \Cref{C*2OPA}); here, we do have a standard form, where the $ v_j $'s are connected straight up to the top of the external rectangle by strings without any local maxima or minima.
Let us look at the $ j $-th internal rectangle; suppose $ v_j \in W_\t{alt}$ is of $ (\vlon, -\vlon ) $-type.
By the action of tangles of the oriented planar algebra $ P^{Y_+} $ defined in \Cref{C*2OPA}, the unitary $ \alpha^*_{v_j}  $ appearing in the label $ \gamma_{v_j} x_j = \alpha^*_{v_j} \ (1_{g_{\vlon}} \otimes x_j \otimes 1_{g_{-\vlon}} ) \ R_\vlon $, slides straight up to the top of the external rectangle of $ T $.
Moreover, the $ g_\vlon $ and $ g_{-\vlon} $ at the two extremes of $ Y_{v_j} $ (which are also part of the external rectangle), are pair-partitioned by the $ R_\vlon $ appearing in $ \gamma_{v_j} x_j$.
This lets us to express $ P^{Y_+}_T (\gamma_{v_1} x_1 , \dots , \gamma_{v_m} x_m )$ as $ P^{g_+}_{S''} \otimes P^{X_+}_T (x_1,\ldots , x_m) $ where $ S'' $ is a pair-partitioning of the $ g_\pm $ appearing in $ Y_w $ and thus the proof is complete.
\end{proof}

\subsection{Free oriented extension of the Temperley-Lieb Planar algebra}\label{otl sec}
The foremost example of oriented extensions comes from the Temperley-Lieb planar algebra.
For $ \delta \geq 2 $, $TL^\delta$ will denote the subfactor planar algebra with modulus $ \delta $.

We consider the free oriented Temperley-Lieb planar algebra, $ \mcal{F}(TL^\delta) $.
For a word $ w $ with letters in $ \{+,-\} $, the vector space $\mcal{F}(TL^\delta)_w $ is the complex span of $ w $-tangles without any internal discs and loops, that is, oriented Temperley-Lieb diagrams.
The oriented planar tangles act on $ \mcal{F}(TL^\delta) $ exactly same way as in the ordinary $ TL^\delta$.
As in $ TL^\delta $, we have a double sequence of Jones-Wenzl idempotents $ f^+_n \in TL^\delta_{+n}= \mcal F(TL^\delta)_{(+-)^{2n}}$ and $ f^-_n \in TL^\delta_{-n}=\mcal{F}(TL^\delta)_{(-+)^{2n}}$ for $ n \geq 1 $.
When $ n=0 $, $f^+_0$ and $f^-_0 $ get identified in $ \mcal{F}(TL^\delta) $ which we denote by $ f_0 $.
The projection category of $ \mcal{F}(TL^\delta) $, $ \mcal K(\mcal C^{\mcal{F}(TL^\delta)}) $, is generated by the projection $ f^+_1$ and has $ f_0 $ as the tensor unit.
It will be interesting to look at the irreducible objects of $ \mcal C^\delta_{\t {free}} $.
For this, we use the MAS sub-word decomposition of words with letters in $ \{+,-\} $ (possibly starting and ending with the same sign).
Let $ v=v_1v_2\cdots v_n $ be the MAS sub-word decomposition of $ v $ where each $ v_i $ starts in $ \varepsilon_i \in \{+,-\} $.
If we set $ f_{vv^*} $ to be the projection, 
\[
\psfrag{a}{\hspace{-0.45em}$ f^{\vlon_1}_{\abs {v_1}} $}
\psfrag{b}{\hspace{-0.65em}$ f^{\vlon_2}_{\abs {v_2}} $}
\psfrag{c}{\hspace{-0.5em}$ f^{\vlon_n}_{\abs {v_n}} $}
\psfrag{1}{$ v_1 $}
\psfrag{2}{$ v_2 $}
\psfrag{n}{$ v_n $}
\psfrag{d}{$  \bigstar $ }
\raisebox{-1.9em}{\includegraphics[scale=0.3]{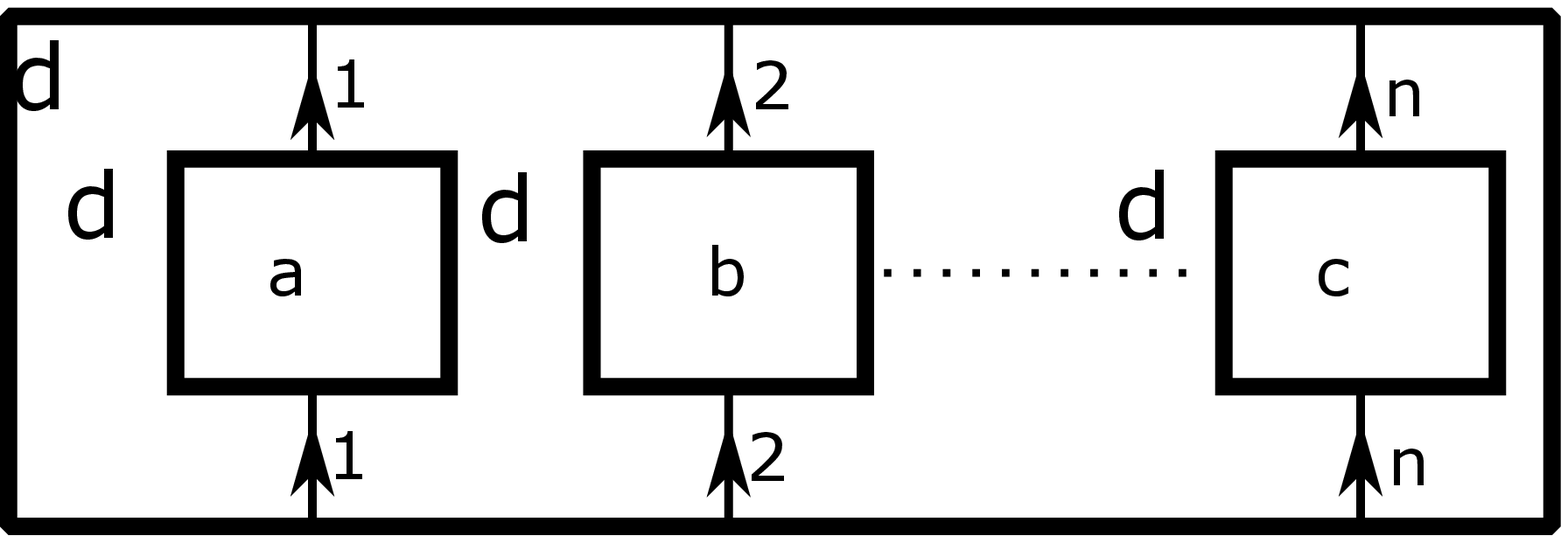}},\]
then, with little effort, one can prove that $ f_{vv^*} $ is a minimal projection of $ \mcal{F}(TL^\delta)_{vv^*} $ and thereby a simple object $ \mcal K(\mcal C^{\mcal{F}(TL^\delta)}) $.
We now need to see whether for all $ w $, every simple object in $ \mcal K(\mcal C^{\mcal{F}(TL^\delta)}) $ (same as a minimal projection in $ \mcal F(TL^\delta)_w $) is equivalent to $ f_0 $ or one of these $ f_{vv^*} $'s.
For projections to exist, we should necessarily have $ w=u\tilde{u} $.
Using the standard trick of `through strings' and `middle pattern analysis' (see \cite{BJ}), one can show that this is indeed the case.

From \cite{B1} and \cite{BRV}, one can see that the (co-)representation category of free unitary quantum group $A_u(F)  $ for $ F\geq 0 $ and $ \mcal F(TL^\delta) $ are equivalent.

One more $ TL^\delta $ are the so called \textit{unshaded Temperley-Lieb} denoted by $ USTL^\delta$.
Define $ USTL^\delta_w $ as the span of set of all non-crossing pairings of letters in $  w $ (irrespective of the signs, that is, pairing of like signs is allowed).
This automatically puts the restriction $ USTL^\delta_w $ is zero if length of $ w$ is odd.
Given any oriented tangle, removing all the directions and labels of the strings gives an unshaded TL tangle and hence can act on $ USTL^\delta $ (action of tangles with any of the discs having color of odd length is taken to be zero).
Note that, irrespective of the sequence of letters in words, if the length of two words are same, then the corresponding vector spaces are identical.
Clearly, $ USTL^\delta $ is a $ \{+\} $-oriented factor planar algebra for $ \delta\geq 2 $ and $ \mcal{F}(TL^\delta) $ sits inside it in a canonical way (as proved in Theorem \ref{free con any}).
Under this inclusion, the projection $ f_{vv^*} $ is no longer minimal if there is at least two MAS sub-words in $ v $.
In fact, the irreducibles in the projection category $ \mcal K(\mcal C^{USTL^\delta}) $ of $ USTL^\delta $, come from those $ f_{vv^*} $ 's for which all the letters in $ v $ are alternately signed.
Now, note that $ USTL^\delta_{++} $ and $ USTL^\delta_{--} $ are one dimensional; this shows $ f_{+-} $ and $ f_{-+} $ are isomorphic and thereby, $ f_{(+-)^k} $ and $ f_{(-+)^k} $ become isomorphic.
Hence simple objects of $ \mcal K(\mcal C^{USTL^\delta}) $ can be identified with $ \N\cup\{0\} $. 

This category can be realized as the representation category of the free orthogonal quantum groups $A_{o}(F)$, where $F$ is a matrix with $\operatorname{Tr}(F^{*}F)=\delta$ and $\overline{F}F=1$. There is another case, namely when $\overline{F}F=-1$, which corresponding to $Rep(SU_{q}(2))$ with $q+q^{-1}=\delta$. These planar algebras cannot actually be ``unshaded" since the generating object is not symmetrically self-dual, but nevertheless provide oriented extensions $TL^{\delta}$ (see \cite{B}).

We propose the following natural problem.

\begin{problem} Classify all oriented extension of the $TL^{\delta}$ for $\delta\ge 2$.
\end{problem}

We expect the corresponding problem for $\delta<2$ to actually be more difficult, and relate closely to the extension theory of fusion categories \cite{ENO}.
The free product of the ``even parts" of this subfactor planar algebra are very likely to have a large number of quotients, and each of these will likely have a large number of invertible bimodules, making classifications of extensions difficult.
However, in the case $\delta\ge 2$, we expect the number quotients of the even part to be manageable, and the number of invertible bimodules to be small, making this problem feasible.
\subsection{Hyperfinite constructions}\label{applications}

In this section, we make use of a result of Vaes about existence and uniqueness of subfactor standard invariants in the hyperfinite $\rm{II}_1$ factor to provide some more examples.
In \Cref{FOE}, we have seen that every subfactor planar algebra has a canonical oriented extension, namely the free one.
However, as described in this introduction, if we know that there exists a hyperfinite subfactor whose standard is given by the subfactor planar algebra (which we refer as \textit{hyperfinite realizable}), we can produce many oriented extensions.
We describe the procedure below.

Suppose $ N\subset M $ is an extremal, finite index subfactor such that there is an isomorphism $ \vphi :N \ra M $.
Consider the extremal bifinite $ N $-$ N $ bimodule $  \mcal H^\vphi $ given by: $ \mcal H^\vphi := L^2 (M) $ and $ n_1 \cdot \hat{x} \cdot n_2 :=\widehat{ \left[n_1 y \vphi (n_2)\right]} $\; for all $ x \in M $ and $ n_1,n_2 \in N $.
As discussed in \Cref{C*2OPA}, we can associate a singly generated oriented planar algebra, say $ OP^\vphi $, to the rigid $ C^* $-tensor category generated by $ \mcal H^\vphi $.
The shaded part of $ OP^\vphi $ indeed turns out to be isomorphic to the subfactor planar algebra $ P^{N\subset M} $ associated to $ N\subset M $; thereby, $ OP^\vphi $ becomes an oriented extension of $ P^{N\subset M} $.
To see this, one has to use the isomorphism between the grid of relative commutants and intertwiner spaces (as in \cite{JS}) and the decomposition $\vphantom{ \mcal H}^{\vphantom{ \vphi}}_N  \mcal H^\vphi_N \  \cong \ \vphantom{L^2(M)}_N L^2 (M)_M \ \us M \otimes \ \vphantom{ \mcal H}^{\vphantom{ \vphi}}_M  \mcal H^\vphi_N$ (where $ \vphantom{ \mcal H}^{\vphantom{ \vphi}}_M  \mcal H^\vphi_N $ is an invertible bimodule); a more explicit  isomorphism can be found in \cite{MB} or \cite[Theorem 5.2]{DGG}.
In other words, starting from a subfactor planar algebra $ P $ such that it is known that the $ P $ comes from a subfactor $ N \subset M $ where $ N $ and $ M $ are isomorphic, then every isomorphism between $ N $ and $ M $ gives rise to an oriented extension of $ P $.
In particular, if the subfactor planar algebra corresponds to a hyperfinite one, then one can easily obtain many of its oriented extensions by picking different isomorphisms from $N$ to $M$.

We next deal with the question whether the free oriented extension of a hyperfinite realizable subfactor planar algebra is hyperfinite realizable. For this, recall the following definitions from \cite{V} regarding freeness of two fusion subalgebras of bifinite bimodules over a $ \rm{II}_1 $ factor.

\begin{defn}[Vaes]\label{fusfreedef}
Let $ M $ be a $ \rm{II}_1 $-factor and $ \mcal F_1,\mcal F_2 $ be two fusion subalgebras of the fusion algebra of bifinite bimodules over $ M $ with basis $ \chi_1 $ and $ \chi_2 $ respectively.
Then, $ \mcal F_1 $ and $ \mcal F_2 $ are said to be \textit{free} if:
\begin{itemize}
	\item [(i)] every tensor product of non-trivial irreducible bimodules, with factors alternatingly from $ \chi_1 $ and $ \chi_2 $, is irreducible,
	\item [(ii)] two tensor products of non-trivial irreducible bimodules, with factors alternatingly from $ \chi_1 $ and $ \chi_2 $, are equivalent if and only if they are factor by factor equivalent.
\end{itemize}
\end{defn}
\begin{prop}
The free oriented extension of a hyperfinite realizable subfactor planar algebra is isomorphic to a oriented factor planar algebra singly generated  by an extremal bifinite bimodule over the hyperfinite $ \rm{II}_1 $-factor $ R $.
\end{prop}
\begin{proof}
Suppose $ \t{Bim}_{\t{ext}}  (R)$ denotes the category of extremal bifinite bimodules over $ R $.
Since subfactor planar algebras are assumed to be spherical, without loss of generality, we may start with a subfactor planar algebra $ P $ associated to $\vphantom{\mcal H}_R \mcal H_R \in \t {Obj} (\t{Bim}_{\t{ext}}  (R))$.
Let $ Q $ be the singly generated oriented factor planar algebra associated to $\vphantom{\mcal H}_R \mcal H_R$, and $ \mcal C $ be the full subcategory of $ \t{Bim}_{\t{ext}}  (R) $, tensor-generated by $\vphantom{\mcal H}_R \mcal H_R$.
From \Cref{OPA-C*-correspondence}, $ \mcal C $ is monoidally equivalent to the projection category $ \mcal K(\mcal C^Q) $ (associated to $ Q $) as C*-categories.

Consider an outer action $ \kappa $ of $ \Z $ on $ R $.
For $ n \in \Z $, let $\vphantom{\mcal K_n}_R {\mcal K_n}_R $ be the invertible bimodule $ L^2 (\kappa_n) $, that is, $ \mcal K_n \coloneqq L^2 (R) $ on which the left (resp., right) action of $ R $ is the usual one (resp., twisted by $ \kappa_n $).
Suppose $ \mcal D $ denotes the subcategory of $ \t{Bim}_{\t{ext}}  (R) $, whose irreducible sub-modules are isomorphic to $ \mcal K_n $'s.
Clearly, $ \mcal D $ is equivalent to $ \operatorname{Hilb}_{f.d.}(\Z )$.

Let $ \mcal G_{\mcal C} $ and $ \mcal G_{\mcal D}$ be the fusion sub-algebras of the fusion algebra of $ \t{Bim}_{\t{ext}}  (R) $ corresponding to $ \mcal C $ and $ \mcal D $ respectively.
Now, \cite[Theorem 5.1]{V} tells us that there exists $ \theta \in \t{Aut} (R)$ such that the fusion sub-algebras corresponding to $ \mcal C $ and $ L^2 (\theta) \us R \otimes \mcal D \us R \otimes \ol{L^2(\theta)}$ are free.
So, replacing the outer action $ \kappa $ by $ \t {Ad}\theta \circ \kappa $, we may assume that $ \mcal G_{\mcal{C}} $ and $ \mcal G_{\mcal{D}} $ are free.
We claim that the full subcategory $ \mcal E $ of $  \t{Bim}_{\t{ext}}  (R) $ tensor-generated by $ \mcal C $ and $ \mcal D $ is monoidally equivalent to $ \mcal C \ast \mcal D $.

Applying \Cref{freeprodunilem}, we get a monoidal C*-functor $ A:\mcal C \ast \mcal D \ra \mcal E $ such that the restriction of $ A $ to $ \mcal C $ (resp., $ \mcal D $) is equivalent to the containment of $ \mcal C $ (resp., $ \mcal D $) in $ \mcal E $ as a full subcategory.
Since $ \mcal C \ast \mcal D $ and $ \mcal E $ are semi-simple, rigid C*-tensor categories, $ A $ being a monoidal C*-functor, must be faithful.
\Cref{fusfreedef} (i) and \Cref{freeprod} implies that $A $ must send simple objects to simple ones; \Cref{fusfreedef} (ii) implies that $ A$ must induce a bijection on isomorphism classes of simple objects. Any monoidal C*-functor between semi-simple rigid C*-tensor categories which induces a bijection between isomorphism classes of simple objects is an equivalence.

Thus we have a copy of $\mcal{C}*\operatorname{Hilb}_{f.d.}(\Z )$ as a full subcategory of $\t{Bim}_{\t{ext}}(R)$, so the result follows by Theorem \ref{orientedrealizationivertible}.
\end{proof}

Given a subfactor planar algebra, it would be interesting to find out all possible oriented extensions (up to isomorphism).

\bibliographystyle{alpha}

\end{document}